\documentclass[12pt]{article}

\usepackage[margin=2cm]{geometry}

\usepackage{amsmath,amssymb,amsthm,latexsym}

\usepackage[all]{xy}

\theoremstyle{plain}
\newtheorem{proposition}{Proposition}[section]
\newtheorem{theorem}[proposition]{Theorem}
\newtheorem{corollary}[proposition]{Corollary}
\newtheorem{lemma}[proposition]{Lemma}
\theoremstyle{definition}
\newtheorem{definition}[proposition]{Definition}
\newtheorem{example}[proposition]{Example}
\newtheorem{remark}[proposition]{Remark}

\newcommand{\Z}{\mathbb Z}

\newcommand{\ip}[2]{\langle #1,#2 \rangle}
\newcommand{\wap}{{\operatorname{WAP}}}
\newcommand{\aone}{\Box}
\newcommand{\atwo}{\Diamond}
\newcommand{\id}{\mathrm{id}}
\newcommand{\BZ}{\beta\Z}
\newcommand{\Uf}{\mathcal U}
\newcommand{\npU}{\mathbb Z^*}
\renewcommand{\S}{\mathcal S}
\newcommand{\T}{\mathcal T}

\newcommand{\la}{\langle}
\newcommand{\ra}{\rangle}

\numberwithin{equation}{section}

\begin{document}

\title{Shift invariant preduals of $\ell_1(\Z)$}
\author{Matthew Daws\\\normalsize{\texttt{matt.daws@cantab.net}}\and Richard Haydon\\\normalsize{\texttt{richard.haydon@bnc.ox.ac.uk}}\and Thomas Schlumprecht\thanks{
 Supported by  NSF grants DMS0856148 and DMS0556013.}\\\normalsize{\texttt{schlump@math.tamu.edu}}\and Stuart White\\\normalsize{\texttt{stuart.white@glasgow.ac.uk}}}
\maketitle

\begin{abstract}
The Banach space $\ell_1(\Z)$ admits many non-isomorphic preduals, for example,
$C(K)$ for any compact countable space $K$, along with many more exotic Banach spaces.
In this paper, we impose an extra condition: the predual must make the bilateral
shift on $\ell_1(\Z)$ weak$^*$-continuous.  This is equivalent to
making the natural convolution multiplication on $\ell_1(\Z)$
separately weak$*$-continuous and so turning $\ell_1(\Z)$ into a dual
Banach algebra.  We call such preduals \emph{shift-invariant}.
It is known that the only shift-invariant predual arising from the standard duality
between $C_0(K)$ (for countable locally compact $K$) and $\ell_1(\Z)$ is $c_0(\Z)$.
We provide an explicit construction of an uncountable family of distinct preduals which
do make the bilateral shift weak$^*$-continuous.  Using Szlenk index arguments, we show
that merely as Banach spaces, these are all isomorphic to $c_0$.
We then build some theory to study such preduals, showing that
they arise from certain semigroup compactifications of $\Z$.  This allows us to produce
a large number of other examples, including non-isometric
preduals, and preduals which are not Banach space isomorphic to $c_0$.
\end{abstract}

\section{Introduction}

The Banach space $\ell_1(\Z)$ has a multitude of preduals beyond the canonical pairing between $c_0(\Z)$ and $\ell_1(\Z)$.  For example,
if $X$ is any countable, compact Hausdorff space, then $C(X)^* = M(X) =
\ell_1(X) \cong \ell_1(\Z)$ as all measures are countably additive.
However, preduals of $\ell_1$ can be very exotic.  In \cite{bl}, it was shown that
there exist isometric preduals of $\ell_1$ which are not isomorphic to a complemented
subspace of any $C(K)$ space.  In \cite{bd}, a predual $Y$ of $\ell_1(\Z)$ was constructed
such that $Y$ has the Radon-Nikodym property and each infinite-dimensional subspace of
$Y$ contains a further infinite-dimensional subspace which is reflexive.  This
construction was an inspiration for the recent solution to the scalar-compact problem
\cite{ah}: this exotic Banach space is also an $\ell_1$ predual.  Indeed, in \cite{fos},
it is shown that if $X$ is any Banach space with separable dual, then
there is an $\ell_1$ predual $E$ which contains an isomorphic copy of $X$.
In this paper, we do not assume that a predual $E$ of  $\ell_1(\Z)$ is isometric, and
instead we allow any isomorphism between $E^*$ and $\ell_1(\Z)$.

Every predual of $\ell_1(\Z)$ can be canonically regarded as a subspace $E$ of $\ell_\infty(\Z)$, albeit in a possibly non-isometric fashion.  This paper addresses the question of which preduals are invariant under the bilateral shift operator on $\ell_\infty(\Z)$.  Equivalently, this asks which preduals make the bilateral shift operator on $\ell_1(\Z)$ weak$^*$-continuous.  Clearly $c_0(\Z)$ is one such predual, but we are interested in the existence
of other preduals: by results of \cite{dlpw} these are necessarily slightly
exotic (see the discussion at the end of Section~\ref{prelim} below). In particular, given a countable, compact Hausdorff space, the canonical duality between $C(X)$ and $M(X)=\ell_1(X)\cong\ell_1(\Z)$ cannot make the bilateral shift operator weak$^*$-continuous.

Our interest in this topic is motivated by Banach algebra theory.  The Banach space $\ell_1(\Z)$ becomes a Banach algebra for the convolution product:
\begin{equation}
(f*g)(n) = \sum_{k\in\Z} f(k) g(n-k), \quad f,g\in\ell_1(\Z),\ n\in\Z.
\end{equation}
A Banach algebra is a \emph{dual Banach algebra} if it is a dual space of some Banach space and the product is separately weak$^*$-continuous, see \cite{Runde2}.  In particular $\ell_1(\Z)$ is a dual Banach algebra when equipped with the standard predual $c_0(\Z)$.  The standard warning in the theory of dual Banach algebras is that, unlike the situation with von Neumann algebras, the predual need not be unique: indeed, give $\ell_1$ the zero product, so that any predual turns $\ell_1$ into a dual Banach algebra.  However, there has been little investigation of what happens in natural classes of Banach algebras; see Section~\ref{prelim} for further details.  Motivated by Sakai's classical work on the preduals of von Neumann algebras, the first named author asked in \cite{daws} whether the weak$^*$-topology induced by $c_0(\Z)$ is the unique way of turning $\ell_1(\Z)$ into a dual Banach algebra.  The results of this paper answer this question negatively: preduals on the
convolution algebra $\ell_1(\Z)$ are far from unique.  An easy calculation (see
Proposition~\ref{equivalentforms} below) shows that a predual for $\ell_1(\Z)$ makes
the multiplication separately weak$^*$-continuous if, and only if, it is shift-invariant 
regarded as a concrete subspace of $\ell_\infty(\Z)$.

We aim to investigate these preduals from both the Banach algebra and Banach space
viewpoint.  From the algebra viewpoint, our focus is on exotic weak$^*$-topologies making $\ell_1(\Z)$ into a dual Banach algebra. For shift-invariant preduals for $\ell_1(\Z)$, we examine possible limit points of the set of point masses.  From the Banach space viewpoint, we initiate the Banach space classifcation of shift-invariant preduals.  It is important to note that two shift-invariant preduals may be isomorphic as Banach spaces, yet induce very different weak$^*$-topologies, so these two viewpoints ask quite different questions about our predual.  Although it does not really matter in this paper, we work with complex scalars throughout.

In Section \ref{haydon} we construct a non-canonical shift-invariant predual.  This predual is defined to be the closed linear span $E$ in $\ell_\infty(\Z)$ of bilateral shifts of the element
\begin{equation}
x_0=(\cdots0\ 0\ 1\ 2^{-1}\ 2^{-1}\ 2^{-2}\ 2^{-1}\ 2^{-2}\ 2^{-2}\ 2^{-3}\ 2^{-1}\ \cdots),
\end{equation}
where the $1$ appears in the zero'th component of $x_0$ and, for $n>0$, the number of $1$'s in the binary expansion of $n$ determine the negative exponent of $2$ in $x_0(n)$. We give a direct proof that $E$ provides a predual of $\ell_1(\Z)$, which also explicitly describes those elements of $\ell_\infty(\Z)\cong C(\BZ)$ which lie in $E$.  With respect to this predual, $\delta_{2^n}\rightarrow \delta_0/2$ in the weak$^*$-topology; indeed it is easily seen that for all $m\in\Z$, $x_0(2^n+m)\rightarrow x_0(m)/2$ as $n\rightarrow\infty$.  However from the Banach space prospective, $E$ is isomorphic to $c_0$. We demonstrate this by using Benjamini's work on $G$-spaces to observe that $E$ is a $C(K)$ space for some countable compact $K$ (though of course the duality between $E$ and $\ell_1(\Z)$ is not obtained via the canonical identification of $C(K)$ as a predual of $\ell_1(\Z)$) and then calculating the Szlenk index of $E$.

In Section \ref{semigroup}, we work more abstractly, developing a general framework for the study of shift-invariant preduals in terms of compact semigroup compactifications of $\Z$. We show in Theorem \ref{char-preduals} that every shift-invariant predual of $\ell_1(\Z)$ is the preannihilator of the kernel of a bounded homomorphism $\Theta:M(\S)\rightarrow\ell_1(\Z)$ which is also a projection for some suitable semigroup compactification $\S$ of $\Z$. This machinery enables us to quickly construct a variety of new preduals. In Section \ref{eg} we give examples of how this can be done and show that the example described in the previous paragraph also fits into this setting.  We are able to produce preduals by adding finitely many exotic weak$^*$-limit points of the point masses, such as the limit $\delta_{2^n}\rightarrow \delta_0/2$ appearing in our previous predual.  In particular, given $a_1,\cdots,a_k$ in $\ell_1(\Z)$ and disjoint infinite sets $J^{(1)},\cdots,J^{(k)}$ in $\Z$ we are able to produce a shift-invariant predual for which $\delta_n\rightarrow a_i$ as $|n|\rightarrow\infty$ through the set $J^{(i)}$ provided:
\begin{itemize}
\item The $a_i$ are power bounded in $\ell_1(\Z)$ (i.e. $\sup_m\|a_i^m\|_1<\infty$) and convolution powers become uniformly small (i.e. $\|a_i^m\|_\infty\rightarrow0$ as $m\rightarrow\infty$);
\item The sets $J^{(i)}$ are suitably sparse in a sense that will be made precise later.
\end{itemize}
We use the approach to construct shift-invariant preduals which are not isomorphic as Banach spaces to $c_0$ (see Theorem~\ref{not_co_thm}) and shift-invariant preduals which are not isometrically induced (see Example~\ref{not_iso_eg}).

It is also possible replace $\Z$ by any countable discrete group $G$ (or even a semigroup) and ask for dual Banach algebra preduals of $\ell_1(G)$ other than $c_0(G)$. The work of \cite{dlpw} applies in this context, and shows that such preduals cannot be obtained by the canonical duality between $C(X)$, for a countable, compact Hausdorff space $X$ and $M(X)\cong \ell_1(G)$.  We do not pursue arbitrary groups here, as even in the case of $\Z$, which has a very simple algebraic structure, the construction of shift-invariant preduals is somewhat involved.  In the semigroup context, however, it can be much easier to produce such preduals: see \cite{semipaper} for a discussion of shift-invariant preduals on $\Z\times\Z^+$.

\paragraph*{Acknowledgements} This paper was iniatated during visits of Matt Daws and
Richard Haydon to Thomas Schlumprecht and Stuart White at Texas A\&{}M University.
Matt Daws and Richard Haydon would like to thank the faculty at Texas A\&{}M for their
hospitality.  Matt Daws and Stuart White worked on this paper while attending a workshop
at the University of Leeds which was supported by EPSRC grant EP/I002316/1.

\section{Shift-invariant preduals}\label{prelim}

A \emph{dual Banach algebra} is a Banach algebra which is also a dual Banach space,
such that the product is separately weak$^*$-continuous.  The term was introduced
in \cite{Runde2}, but the concept had been studied before, see \cite[Section~4]{kai}
or \cite{young}.  A C$^*$-algebra $M$ which is \emph{isometric} to a dual space is a
W$^*$-algebra, and then the product, and the involution, are automatically 
weak$^*$-continuous, and $M$ can be weak$^*$-represented on a Hilbert space, that is,
$M$ is a von Neumann algebra, see \cite{sakai}.  Furthermore, in this case, the
predual of $M$ is unique, isometrically.  However, Pe{\l}czy\'nski showed in
\cite{pel} that $\ell_\infty$ and $L_\infty[0,1]$ are isomorphic as Banach spaces
(but not isometrically isomorphic), while of course $\ell_1$ and $L_1[0,1]$ are not
isomorphic.  Thus the predual of a von Neumann algebra is not isomorphically unique.
Authors Matt Daws and Stuart White together with Hung Le Pham showed in
\cite[Theorem~5.2]{dlpw}
that a Banach algebra isomorphism (not necessarily isometric) between a von Neumann algebra
and a dual Banach algebra is always weak$^*$-continuous.  For further discussion of
the uniqueness of preduals for dual Banach algebras, see \cite{dlpw, semipaper}.

The normal cohomology (that is, topological cohomology taking account of the
weak$^*$-topology) of von Neumann algebras has been extensively studied,
see \cite{ss2,cpss} for example.
Runde was interested in the dual Banach algebra version of this theory in \cite{Runde2}.
For example, he showed in \cite{runde} that for a locally compact group $G$,
the first weak$^*$-continuous cohomology for $M(G)$ with values in a normal bimodule is
trivial if and only if $G$ is \emph{amenable}.  If we do not take account of the
weak$^*$-topology, then $G$ is forced to be discrete as well, \cite{DGH}.
Of course, here we have to specify the canonical predual $C_0(G)$.  It would be
interesting to know how varying the predual (if possible) affects the cohomological
properties of $M(G)$.

We write $\ip{\cdot}{\cdot}$ for the bilinear pairing between a Banach space and its dual.
Given a closed subspace $F\subseteq\ell_\infty(\Z)$, the dual space $F^*$ is canonically
isometrically isomorphic to $\ell_\infty(\Z)^* / F^\perp$ where $F^\perp=\{\Phi
\in\ell_\infty(\Z)^* : \ip{\Phi}{x}=0,\ \forall x\in F\}$. Let $\iota_F:\ell_1(\Z)\rightarrow F^*$ be the composition of the canonical embedding
$\kappa_{\ell_1(\Z)}:\ell_1(\Z)\rightarrow\ell_1(\Z)^{**}=\ell_\infty(\Z)^*$ with the
restriction map $\ell_\infty(\Z)^*\rightarrow F^*$.  Thus $\ip{\iota_F(a)}{x} =
\ip{x}{a}$ for $a\in\ell_1(\Z)$ and $x\in F$.
We will say that such an $F$ is a \emph{concrete predual} for $\ell_1(\Z)$ if the map $\iota_F$ is an isomorphism (which is not assumed to be isometric). The next lemma shows that we lose nothing by working with these concrete preduals, and so henceforth we shall do so.

\begin{lemma}
Let $E$ be a Banach space and $\theta:\ell_1(\Z)\rightarrow E^*$ be an isomorphism.  Then the map $\theta^*\kappa_E:E\rightarrow\ell_\infty(\Z)$ is an isomorphism
onto its range, say $F\subseteq\ell_\infty(\Z)$.  Furthermore, $\iota_F$ is an isomorphism so that $F$ is a concrete predual for $\ell_1(\Z)$ and the weak$^*$-topologies induced by the pairings $(\ell_1(\Z)\stackrel{\theta}{\cong}E^*,E)$ and $(\ell_1(\Z),F)$ agree.
That is, given a net $(a_\alpha)$ in $\ell_1(\Z)$, we have that $\lim_\alpha \ip{\theta(a_\alpha)}{x}=0$ for all $x\in E$ if and only if $\lim_\alpha \ip{y}{a_\alpha}=0$ for all $y\in F$.
\end{lemma}
\begin{proof}
Let $T=\theta^*\kappa_E:E\rightarrow\ell_\infty(\Z)$.  Since, for $a\in\ell_1(\Z)$ and
$x\in E$,
\begin{align} \ip{T^*\kappa_{\ell_1(\Z)}(a)}{x} = \ip{T(x)}{a}
= \ip{\theta^*\kappa_E(x)}{a} = \ip{\theta(a)}{x}, \end{align}
it follows that $T^* \kappa_{\ell_1(\Z)} = \theta$.
So, for $x\in E$,
\begin{align} \|T(x)\| &= \sup\{ |\ip{T(x)}{a}| : a\in\ell_1(\Z), \|a\|\leq1 \} \nonumber \\
&= \sup\{ |\ip{\theta(a)}{x}| : a\in\ell_1(\Z), \|a\|\leq1 \}
\geq \frac{\|x\|}{\|\theta^{-1}\|}.
\end{align}
As $T$ is bounded below, we regard $T$ as being an isomorphism onto its range $F$.  Then, for $a\in\ell_1(\Z)$ and $x\in E$,
\begin{align} \ip{T^* \iota_F(a)}{x} &= \ip{T(x)}{a} = \ip{\theta(a)}{x}, \end{align}
so that $T^* \iota_F = \theta$.  Hence $\iota_F = (T^*)^{-1} \theta$ is an isomorphism,
and so $F$ is a concrete predual of $\ell_1(\Z)$.

A net $(a_\alpha)$ in $\ell_1(\Z)$ is null for the $(\ell_1(\Z),F)$ topology if
and only if
\begin{equation} 0 = \lim_\alpha \ip{T(x)}{a_\alpha} = \lim_\alpha \ip{\theta(a_\alpha)}{x},
\qquad x\in E. \end{equation}
That is, if and only if $(\theta(a_\alpha))$ is weak$^*$-null in $E^*$, as required.
\end{proof}

It is easily checked (see \cite[Proposition 2.2]{dlpw}) that in the situation above, $\theta$ is isometric if and only if $\iota_F$ is isometric.  In this case we say that the predual is an \emph{isometric} predual of $\ell_1(\Z)$. The setting of concrete preduals also enables us to easily detect whether two preduals $F_1,F_2\subset\ell_\infty(\Z)$ induce the same weak$^*$-topology on $\ell^1(\Z)$. This happens if, and only if, they are equal as subspaces.

\begin{lemma}\label{concreteequal}
Let $E_1$ and $E_2$ be preduals of $\ell_1(\Z)$, and use these to induce concrete preduals
$F_1,F_2 \subseteq \ell_\infty(\Z)$ as above.  Then $E_1$ and $E_2$ induce the same
weak$^*$-topology on $\ell_1(\Z)$ if and only if $F_1=F_2$.
\end{lemma}
\begin{proof}
It is immediate from the previous lemma that $E_1$ and $E_2$ induce the same weak$^*$-topology when $F_1=F_2$.  Conversely, for $i=1,2$, let $\theta_i:\ell_1(\Z)\rightarrow E_i^*$ be an isomorphism, and suppose that these induce the same weak$^*$-topology
on $\ell_1(\Z)$.  Towards a contradiction, suppose there exists $x\in F_2\setminus F_1$.
By Hahn-Banach, there exists $\Lambda\in \ell_\infty(\Z)^*$ with $\ip{\Lambda}{x}=1$
and $\ip{\Lambda}{y}=0$ for each $y\in F_1$.  Let $(a_\alpha)$ be a bounded net in
$\ell_1(\Z)$ which converges to $\Lambda$ weak$^*$ in $\ell_\infty(\Z)^*$.  Then
$\lim_\alpha \ip{y}{a_\alpha}=0$ for $y\in F_1$, so by the previous lemma,
$(\theta_1(a_\alpha))$ is weak$^*$-null in $E_1^*$.  By assumption, it follows
that $(\theta_2(a_\alpha))$ is weak$^*$-null in $E_2^*$, but this contradicts that
$1 = \ip{\Lambda}{x} = \lim_\alpha \ip{x}{a_\alpha}$, as $x\in F_2$.  This shows that
$F_2 \subseteq F_1$, and analogously, $F_1 \subseteq F_2$, as required.
\end{proof}
\noindent In a similar vein to the lemma above, concrete preduals $F_1\subseteq F_2\subseteq \ell_\infty(\Z)$ must be equal. Of course, it is possible that 
preduals $F_1,F_2\subseteq\ell_\infty(\Z)$ inducing different weak$^*$-topologies are isomorphic as Banach spaces. Examples of this phenomena will be given in Section \ref{haydon}.

We now turn to the preduals which interest us in this paper. We call a predual satisfying the equivalent conditions of the following easy proposition \emph{shift-invariant}.

\begin{proposition}\label{equivalentforms}
Let $F\subseteq\ell_\infty(\Z)$ be a concrete predual for $\ell_1(\Z)$.  Then the following are equivalent:
\begin{enumerate}
\item The bilateral shift on $\ell_1(\Z)$ is weak$^*$-continuous, with respect to $F$;
\item The subspace $F$ is invariant under the bilateral shift on $\ell_\infty(\Z)$;
\item $\ell_1(\Z)$ is a dual Banach algebra, with respect to $F$.
\end{enumerate}
\end{proposition}
\begin{proof}
Let $\delta_1\in\ell_1(\Z)$ be the unit point mass at $1$.  Then convolution by
$\delta_1$ induces the bilateral shift on $\ell_1(\Z)$, and under the convolution
product, $\delta_1$ generates the commutative Banach algebra $\ell_1(\Z)$.
It follows that conditions (1) and (3) are equivalent.

Let $\sigma$ be the bilateral shift on $\ell_1(\Z)$, so that $\sigma^*$ is the bilateral shift (going in the other direction) on $\ell_\infty(\Z)$.  If (1) holds but (2) does not,
we can find $x\in F\setminus \sigma^*(F)$.  So $(\sigma^*)^{-1}(x)\not\in F$, and so
by Hahn-Banach, we can find $\Phi\in\ell_\infty(\Z)^*$ with $\ip{\Phi}{(\sigma^*)^{-1}(x)}=1$ and $\ip{\Phi}{y}=0$ for all $y\in F$.
Pick a bounded net $(a_\alpha)\subseteq\ell_1(\Z)$
which converges weak$^*$ to $\Phi$.  Then $\lim_\alpha\ip{y}{a_\alpha}=0$ for
$y\in F$, so that $(a_\alpha)$ is weak$^*$-null for the weak$^*$-topology given by $F$.
Hence also $(\sigma^{-1}(a_\alpha))$ is weak$^*$-null, so as $x\in F$,
\begin{equation} 0 = \lim_\alpha \ip{x}{\sigma^{-1}(a_\alpha)}
= \lim_\alpha \ip{(\sigma^*)^{-1}(x)}{a_\alpha}
= \ip{\Phi}{(\sigma^*)^{-1}(x)} = 1, \end{equation}
giving the required contradiction.
A similar argument holds if $x\in \sigma^*(F)\setminus F$.  Thus (1)
implies (2). Conversely, when (2) holds, let $(a_\alpha)$ be a weak$^*$-null net in $\ell_1(\Z)$ and let $x\in F$.  Then $\sigma^*(x)\in F$ and so $\ip{x}{\sigma(a_\alpha)}\rightarrow 0$.
So $(\sigma(a_\alpha))$ is weak$^*$-null, showing (1).
\end{proof}

As well as the convolution product, $\ell_1(\Z)$ admits a natural coproduct:
\begin{equation}
\Gamma:\ell_1(\Z)\rightarrow\ell_1(\Z\times\Z),\quad \delta_n\mapsto\delta_{(n,n)}.
\end{equation}
Given a predual $F\subset\ell_\infty(\Z)$ for $\ell_1(\Z)$, the Banach space injective tensor product $F\check\otimes F$ gives an associated predual for $\ell_1(\Z\times\Z)$ (see \cite[Proposition 3.2]{dlpw} for details) and it is natural to ask which preduals make $\Gamma$ weak$^*$-continuous.  In the same vein as the previous proposition, this can be characterised algebraically. Indeed, \cite[Lemma 3.3]{dlpw} shows that $\Gamma$ is weak$^*$-continuous with respect to $F$ if and only if $F$ is a subalgebra of $\ell_\infty(\Z)$ (with the pointwise multiplication).  Then \cite[Theorem~3.6]{dlpw} shows that if $F\subset\ell_\infty(\Z)$ is a predual making both the multiplication and comultiplication weak$^*$-continuous, then necessarily $F=c_0(\Z)$, i.e. the canonical weak$^*$-topology is the unique topology making all the natural operations suitably continuous.  In particular, given a countable compact Hausdorff space $X$, we have a natural pairing between $C(X)$ and $\ell_1(\Z)\cong M(X)=\ell_1(X)$, and following through the isomorphisms involved in exhibiting $C(X)$ as a concrete predual, we obtain a subalgebra of $\ell_\infty(\Z)$.  As such \cite{dlpw} prevents these pairings from providing new shift-invariant preduals of $\ell_1(\Z)$, though as we will see, with other pairings even $c_0(\Z)$ can be used to give many different shift-invariant preduals of $\ell_1(\Z)$.  Note too that the pairings between these $C(X)$ and $\ell_1(\Z)$ resolve the ``co-version'' of the problem under consideration (namely exhibit non-canonical preduals making the comultiplication continuous). It is a little surprising that it is much easier to make the coproduct on $\ell_1(\Z)$
weak$^*$-continuous, than it is to make the product weak$^*$-continuous.

\section{An explicit construction}\label{haydon}

In this section we give a direct construction of an uncountable family $(F_\lambda)_{|\lambda|>1}$ of ``exotic'' shift-invariant preduals of $\ell_1(\Z)$. As subspaces of $\ell_\infty(\Z)$ they are pairwise distinct, and distinct from $c_0$, so induce an uncountable family of distinct weak$^*$-topologies making $\ell_1(\Z)$ into a dual Banach algebra. 

Fix $\lambda\in\mathbb C$ with $|\lambda|>1$. For $n\geq 0$ in $\Z$, let $b(n)$ be the number of ones in the binary expansion of $n$, so $b(1)=1, b(2)=1, b(3)=2, b(4)=1$ and so forth. For $n<0$, set $b(n)=-\infty$.  Define an element $x_0\in \ell_\infty(\Z)$ by $x_0(n)=\lambda^{-b(n)}$, with the convention that $\lambda^{-\infty}=0$. Thus $x_0$ is given by
\begin{equation}
x_0=(\cdots0\ 0\ 1\ \lambda^{-1}\ \lambda^{-1}\ \lambda^{-2}\ \lambda^{-1}\ \lambda^{-2}\ \lambda^{-2}\ \lambda^{-3}\ \lambda^{-1}\ \cdots)
\end{equation}
where the $1$ occurs in the $n=0$ position of $\mathbb Z$. Let $F$ be the closed shift-invariant subspace of $\ell_\infty(\Z)$ generated by $x_0$, i.e. the closed linear span of the bilateral shifts of $x_0$.  In Theorem~\ref{thm:haydon}, we will show that these $F$ give preduals of $\ell_1(\Z)$ by demonstrating that the canonical map $\iota_F:\ell_1(\Z)\rightarrow F^*$ is a bijection. When we need to indicate the dependance on $\lambda$, we will write $F^{(\lambda)}$ and $x_0^{(\lambda)}$ respectively. 

Write $\sigma$ for the bilateral shift on $\ell_\infty(\Z)$ so that $\sigma(x)(n)=x(n-1)$ for $x\in\ell_\infty(\Z)$.  As a technical device, we introduce a bounded linear operator $\tau:\ell_\infty(\Z)\rightarrow\ell_\infty(\Z)$, defined by
\begin{equation}\label{hadyon:DefT}
\tau(x)(n)=\begin{cases}x(n/2)&n\text{ even};\\0&n\text{ odd}.\end{cases}
\end{equation}
This has the effect of spreading out $x$, for example
\begin{equation}
\tau(x_0)=(\cdots0\ 0\ 1\ 0\ \lambda^{-1}\ 0\ \lambda^{-1}\ 0\ \lambda^{-2}\ 0\ \lambda^{-1}\ 0\ \lambda^{-2}\ 0\ \lambda^{-2}\ 0\ \lambda^{-3}\ 0\ \lambda^{-1}\ \cdots).
\end{equation}
Note that
\begin{equation}\label{eq:TS=S^2T}
\tau\sigma=\sigma^2\tau.
\end{equation}
Indeed, for $n\in\Z$ even,
\begin{equation}
\tau\sigma(x)(n) = \sigma(x)(n/2) = x(n/2-1) = \tau(x)(n-2) = \sigma^2\tau(x)(n),
\end{equation}
while for $n$ odd,  both sides above are trivially zero. As $k$ tends to infinity, $\tau^k(x_0)$ behaves like $\delta_0$ as a functional on $\ell_1(\Z)$. We shall use this phenomenon to establish the injectivity of $\iota_F$ and so we begin by showing that these $\tau^k(x_0)$ lie in the subspace $F$.

\begin{lemma}\label{p:TInvariant}
With the notation above, $\tau^k(x_0)\in F$ for $k\geq 1$.
\end{lemma}
\begin{proof}
We claim that
\begin{equation}\label{p:TInvariant-e1}
(\id-\lambda^{-1}\sigma)(x_0)(n) = (\lambda-1)\sum_{j=1}^\infty \lambda^{-j} \tau^j(x_0)(n),\quad n\in\Z.
 \end{equation}
For $n<0$, both sides of (\ref{p:TInvariant-e1}) evaluate to zero. At $n=0$, we have $(\id-\lambda^{-1}\sigma)(x_0)(0)=1$, while $\tau^j(x_0)(0)=1$ for all $j$ so that the righthand side of (\ref{p:TInvariant-e1}) sums to $1$.  Fix $n>0$, and write the binary
expansion of $n$ as
\begin{equation} n = \sum_{j=k}^l \varepsilon_j 2^j, \end{equation}
where $(\varepsilon_j)\subseteq\{0,1\}$ and $\varepsilon_k=1$.  It follows that
\begin{equation} n-1 = \sum_{j=0}^{k-1} 2^j + \sum_{j=k+1}^l \varepsilon_j 2^j, \end{equation}
and so $b(n-1) = b(n)-1+k$.
 Since $\tau^j(x_0)(n) = x_0(n)$ for $j\leq k$,
and $\tau^j(x_0)(n) = 0$ for $j>k$, we compute that
\begin{align} (\id-\lambda^{-1}\sigma)(x_0)(n) &= x_0(n) - \lambda^{-1}x_0(n-1)
= \lambda^{-b(n)} - \lambda^{-b(n)-k} \\
&= (1-\lambda^{-k})x_0(n)
= (\lambda-1) \sum_{j=1}^{k} \lambda^{-j}x_0(n)\\
&= (\lambda-1) \sum_{j=1}^{k} \lambda^{-j} \tau^j(x_0)(n) = (\lambda-1) \sum_{j=1}^\infty \lambda^{-j} \tau^j(x_0)(n)
\end{align}
to obtain (\ref{p:TInvariant-e1}) for $n>0$.

Applying $(\id-\lambda^{-1}\tau)$ to (\ref{p:TInvariant-e1}) yields
\begin{equation} \label{eq:one}
 \big(\id-\lambda^{-1}\tau\big)\big( \id - \lambda^{-1} \sigma \big)(x_0)
= (\lambda-1) \Big(\sum_{j=1}^\infty \lambda^{-j} \tau^j(x_0)
   - \sum_{j=2}^\infty \lambda^{-j} \tau^j(x_0) \Big) = \frac{\lambda-1}{\lambda} \tau(x_0). \end{equation}
Then, first solving (\ref{eq:one}) for $(\id-\lambda^{-1}\sigma)(x_0)$, and
then applying (\ref{eq:TS=S^2T}), gives
\begin{align} \big(\id-\lambda^{-1}\sigma\big)(x_0) &= \frac{\lambda-1}{\lambda} \tau(x_0)
   + \lambda^{-1} \tau(\id-\lambda^{-1}\sigma)(x_0)
= \tau\big(\id -\lambda^{-2} \sigma \big)(x_0) \\
&= \big(\id - \lambda^{-2}\sigma^2\big)\tau(x_0).
\end{align}
Now, $\|\lambda^{-2}\sigma^2\| = |\lambda^{-2}|<1$ and so $(\id-\lambda^{-2}\sigma^2)$ is invertible with the standard power-series expansion, and hence
\[ \tau(x_0) = \big(\id - \lambda^{-2}\sigma^2\big)^{-1}\big(\id-\lambda^{-1}\sigma\big)(x_0)
= \sum_{j=0}^\infty \lambda^{-2j} \sigma^{2j} \big(\id-\lambda^{-1}\sigma\big)(x_0). \]
In particular, $\tau(x_0)\in F$, as $F$ is shift-invariant and closed.  Using
this expression, and that $\tau\sigma = \sigma^2\tau$, it is now easy to see that $\tau^k(x_0)\in F$
for all $k\geq 1$.
\end{proof}

\begin{lemma}\label{lem:htrick}
The map $\iota_F:\ell_1(\Z)\rightarrow F^*$ is injective.
\end{lemma}
\begin{proof}
Let $k\in\mathbb N$, so that $\tau^k(x_0)(n)= 0$ if $|n|<2^k$ and $n\not=0$,
while $\tau^k(x_0)(0) = 1$.
Thus, for $a=(a_n)\in\ell_1(\Z)$, we see that $a_0 = \lim_{k\rightarrow\infty}
\ip{\tau^k(x_0)}{a}$.  It follows that, if $\iota_F(a)=0$, then $a_0=0$.  By shift
invariance, we see that if $\iota_F(a)=0$, then $a_n=0$ for all $n\in\Z$, that is, $a=0$,
and so $\iota_F$ is injective.
\end{proof}

We now turn to the surjectivity of $\iota_F$. For this we utilise the Stone-Cech compactification $\BZ$ of $\Z$. We regard $\BZ$ as the space of ultrafilters on $\Z$ and write $\npU = \BZ\setminus\Z$ for the non-principal ultrafilters on $\Z$.
The topology on $\BZ$ has basis
\begin{equation} 
\mathcal O_A = \{ \Uf\in\BZ : A\in\Uf \}, \quad A\subseteq\Z. 
\end{equation}
and, as $\BZ\setminus\mathcal O_A = \mathcal O_{\Z\setminus A}$, these sets are also closed in $\BZ$.  We make the canonical identification of $\ell_\infty(\Z)$ with $C(\BZ)$ by extending elements $x\in\ell_\infty(\Z)$ to $\BZ$ by setting $x(\Uf)=\displaystyle{\lim_{n\rightarrow\Uf}x(n)}$.

For $t\in\Z$ define
\begin{equation}
X^{(1)}_t=\{ \Uf\in\Z^* : \forall m>0,\ \{2^n+t:n>m\}\in\Uf\} .
 \end{equation}
As non-principal ultrafilters cannot contain a finite set, it follows that any non-principal ultrafilter containing $\{2^n+t:n>0\}$ must lie in $X^{(1)}_t$. For $k>1$ and $t\in\Z$, define
\begin{equation}
 X^{(k)}_t = \big\{\Uf\in\Z^* : \forall m>0,\ \{2^{n_1}+\cdots+2^{n_k}+t
: m<n_1<n_2<\cdots<n_k\}\in\Uf \big\}. 
\end{equation}
Each $X^{(k)}_t$ is the intersection of sets of the form $\mathcal O_A\cap\Z^*$ and so these sets are closed.  Write $X^{(\infty)}$ for the complement of $\bigcup_{t,k} X^{(k)}_t$ in $\Z^*$.

\begin{lemma}\label{lem:two}
With the notation above, $\Z^*$ is the
disjoint union of $X^{(\infty)}$ and the sets $X^{(k)}_t$.
\end{lemma}
\begin{proof}
Suppose that $X^{(k)}_s \cap X^{(l)}_t$ is non-empty, and fix $\Uf$ in the
intersection.  This means that for all $n,m>0$,
\begin{equation}
 \{2^{n_1}+\cdots+2^{n_k}+s : n<n_1<\cdots<n_k\} \cap
\{2^{m_1}+\cdots+2^{m_l}+t : m<m_1<\cdots<m_l\}  
\end{equation}
is an element of $\Uf$.
Choose $n=m$ so that $2^n>|s-t|$.
Suppose that for some $n<n_1<\cdots<n_k$ and $m<m_1<\cdots<m_l$, we have
\begin{equation}\label{lem:two-e1}
 2^{n_1}+\cdots+2^{n_k} + (s-t)= 2^{m_1}+\cdots+2^{m_l}.
\end{equation}
Now, $\sum_{j=1}^{l-1} 2^{m_j} \geq 2^{m_1} > 2^m$, and so
\begin{equation}
 2^{m_l} = \sum_{i=1}^k 2^{n_i} + (s-t) - \sum_{j=1}^{l-1} 2^{m_j}
< 2^{n_k+1} + |s-t| - 2^m<2^{n_k+1},
\end{equation}
which implies that $m_l \leq n_k$.  By symmetry, $m_l=n_k$.  We can then cancel $n_k$ and $m_l$ from (\ref{lem:two-e1}) and argue in the same way to see that $k=l$ and that $m_i=n_i$ for all $i$.  Thus also $s=t$.
\end{proof}

We can now complete the proof that the $F^{(\lambda)}$ provide shift-invariant preduals of $\ell_1(\Z)$.  The remaining step is to show that the map $\iota_{F^{(\lambda)}}$ is
surjective.  Our calculations also give rise to an intrinsic characterisation of the elements of $F^{(\lambda)}$.

\begin{theorem}\label{thm:haydon}
$F^{(\lambda)}$ is a shift-invariant predual of $\ell_1(\Z)$, and $F^{(\lambda)}$ consists of those $x\in \ell_\infty(\Z)$ which, under the canonical identification of $\ell_\infty(\Z)$ with $C(\BZ)$, satisfy
\begin{equation}\label{thm:haydon-e1}
x(\Uf)=\begin{cases}\lambda^{-k}x(t),&\Uf\in X^{(k)}_t, k\geq 1, t\in\Z;\\0,&\Uf\in X^{(\infty)}.\end{cases}
\end{equation}
\end{theorem}
\begin{proof}
Let us write $G$ for the closed subspace of $\ell_\infty(\Z)\cong C(\BZ)$ given by the conditions in (\ref{thm:haydon-e1}) and note that $G$ is shift-invariant.
For an ultrafilter $\Uf\in\beta\Z$ and $s\in\Z$, we write $\Uf+s = \{ A+s:A\in\Uf\}$
and note that $\Uf+s\in\Z^*$ if and only if $\Uf\in\Z^*$, and that for some
$t\in\Z$ and $k\in\mathbb N$, we have $\Uf\in X^{(k)}_t$ if and only if $\Uf+s
\in X^{(k)}_{t+s}$.
We first show that $x_0\in G$ so that $F\subseteq G$. For $t\geq 0$ and $n$ sufficiently large, $b(2^n+t)=b(t)+1$ so that 
\begin{equation}
\lim_{n\rightarrow\infty}x_0(2^n+t)=\lambda^{-1}x_0(t).
\end{equation}
Let $t<0$, and write $-t = \sum_{j=0}^k \varepsilon_j 2^j$,
with $(\varepsilon_j)\subseteq\{0,1\}$.  For each $j$, let $\varepsilon'_j =
1-\varepsilon_j$, so that $\sum_{j=0}^k \varepsilon_j 2^j + \sum_{j=0}^k
\varepsilon'_j 2^j = 2^{k+1}-1$, and hence for $n>k+1$,
\begin{equation} 2^n + t
= 2^n - 2^{k+1} + \Big( 2^{k+1} - \sum_{j=0}^k \varepsilon_j 2^j \Big)
= 1 + \sum_{j=k+1}^{n-1} 2^j + \sum_{j=0}^{k} \varepsilon_j' 2^j. \end{equation}
Notice that as not every $\varepsilon_j=0$, there is some $j$ with $\varepsilon'_j=0$.
This ensures that $1+\sum_j \varepsilon'_j 2^j \leq 2^{k+1}-1$, and hence
$b(2^n+t) \geq n-k$,  which gives
\begin{equation} \lim_{n\rightarrow\infty} x_0(2^n+t)
= \lim_{n\rightarrow\infty} \lambda^{-b(2^n+t)} = 0 = x_0(t).
\end{equation}
It follows that $x_0(\Uf)=\lambda^{-1}x_0(t)$ for $\Uf\in X^{(1)}_t$ and $t\in\Z$.   Applying these limits twice gives
\begin{equation}
\lim_{n_1\rightarrow\infty}\lim_{n_2\rightarrow\infty}x_0(t+2^{n_1}+2^{n_2})=\lim_{n_1\rightarrow\infty}\lambda^{-1}x_0(t+2^{n_1})=\lambda^{-2}x_0(t),\quad t\in\Z,
\end{equation}
so that $x_0(\Uf)=\lambda^{-2}x_0(t)$ for $\Uf\in X^{(2)}_t$.  For $k>2$ and $t\in\Z$, arguing in this fashion with $k$-iterated limits shows that $x_0(\Uf)=\lambda^{-k}x_0(t)$ for $\Uf\in X^{(k)}_t$.

We complete the proof that $x_0\in G$ by checking that $x_0(\Uf)=0$ for $\Uf\in X^{(\infty)}$.
If this is not the case, then there exists $\Uf\in X^{(\infty)}$ with $x_0(\Uf)\not=0$.
As $x_0$ takes the values $\{0\} \cup \{\lambda^{-k} : k\geq0\}$ on $\Z$, and this
set has only $0$ as a limit point, it follows that $x_0(\Uf)=\lambda^{-k}$ for some
$k\geq 0$.  As such 
\begin{equation}\label{thm:haydon-e2}
\{ 2^{n_1} + \cdots + 2^{n_k} : n_1 < n_2 < \cdots < n_k \}
=\{ n>0 : b(n)=k \}=\{ n\in\Z : x_0(n) = \lambda^{-k} \}\in\Uf
\end{equation}
As $\Uf\not\in X^{(k)}_0$, there exists $m_1>0$ such that
\begin{equation}\label{thm:haydon-e3}
 \Z\setminus\{ 2^{n_1} + \cdots + 2^{n_k} : m_1<n_1 < n_2 < \cdots < n_k \} \in\Uf. 
 \end{equation}
Intersecting the sets in (\ref{thm:haydon-e2}) and (\ref{thm:haydon-e3}) gives
\begin{equation}
\{ 2^{n_1} + \cdots + 2^{n_k} : n_1 < n_2 < \cdots < n_k, n_1\leq m_1 \} \in\Uf.
\end{equation}
As $\Uf$ is an ultrafilter, there exists a fixed $l_1\in\{1,2,\cdots,m_1\}$ such that
\begin{equation} \{ 2^{l_1} + 2^{n_2} + \cdots + 2^{n_k} : l_1 < n_2 < \cdots < n_k \} \in\Uf. 
\end{equation}
As $\Uf\not\in X^{(k-1)}_{2^{l_1}}$, there exists $m_2>0$ such that
\begin{equation}
\{ 2^{l_1} + 2^{n_2} + \cdots + 2^{n_k} : l_1 < n_2 < \cdots < n_k, n_2\leq m_2 \} \in\Uf.
\end{equation}
We then fix $l_2$, and argue the same way, to eventually conclude that we can find
fixed $l_1<l_2<\cdots<l_{k-1}$ with
\begin{equation}
 \{ 2^{l_1} + 2^{l_2} + \cdots + 2^{l_{k-1}}+2^{n_k} : l_{k-1}< n_k \} \in\Uf. 
\end{equation}
However, this shows that $\Uf \in X^{(1)}_t$ for $t=2^{l_1}+\cdots+2^{l_{k-1}}$, a
contradiction.  Therefore $x_0\in G$ and so $F\subseteq G$.

Since $F\subseteq G$, the canonical map $\iota_F$ is the composition of $\iota_G$ followed by the restriction map from $G^*$ onto $F^*$.  By Lemma~\ref{lem:htrick}, $\iota_F$ is injective and hence so too is $\iota_G$.  We now turn to surjectivity.  Given $\mu\in G^*$, extend $\mu$ via the Hahn-Banach theorem to a element of $M(\BZ)=C(\BZ)^*$.  Lemma~\ref{lem:two} ensures that the sets $X^{(\infty)}$, $(X^{(k)}_t)_{t\in\Z,\ k>0}$ are pairwise disjoint. Therefore, for $x\in G$, we can apply countable additivity and the defining identity (\ref{thm:haydon-e1}) to obtain
\begin{align} \ip{\mu}{x} &= \int_{\beta\Z} x \ d\mu
= \int_{X^{(\infty)}} x \ d\mu + \sum_{t\in\Z} \Big( x(t)\mu(\{t\})
   + \sum_{k=1}^\infty \int_{X^{(k)}_t} x \ d\mu \Big) \\
&= \sum_{t\in\Z} x(t) \Big( \mu(\{t\}) + \sum_{k=1}^\infty \lambda^{-k} \mu(X^{(k)}_t) \Big).
\end{align}
Thus $\ip{\mu}{x}=\ip{x}{a}$ for each $x\in G$, where $a=(a_t)\in\ell_1(\Z)$ is defined by
\begin{equation} a_t = \mu(\{t\}) + \sum_{k=1}^\infty \lambda^{-k} \mu(X^{(k)}_t),
\quad t\in\Z.
\end{equation}
As such $\iota_G$, and hence $\iota_F$, is surjective. By the Open Mapping Theorem, both $\iota_F$ and $\iota_G$ are isomorphisms. Hence both $F$ and $G$ are preduals and $F=G$.
\end{proof}

Since concrete preduals $E_1,E_2\subset\ell_\infty(\Z)$ for $\ell_1(\Z)$ induce the same weak$^*$-topology if and only if $E_1=E_2$, it immediately follows that the family $F^{(\lambda)}$ provide uncountably many distinct weak$^*$-topologies turning $\ell_1(\Z)$ into a dual Banach algebra.

\begin{corollary}
The family $(F^{(\lambda)})_{|\lambda|>1}$ induces a continuum of distinct weak$^*$-topologies on $\ell_1(\Z)$ making the multiplication separately weak$^*$-continuous.  Each of these topologies is distinct from the topology induced by the canonical pairing of $c_0$ with $\ell_1(\Z)$.
\end{corollary}
\begin{proof}
This amounts to noting that $x_0^{(\lambda)}\not\in c_0$ for all $\lambda$, which is immediate, and that $x_0^{(\lambda)}\in F^{(\mu)}$ if and only if $\lambda=\mu$, a consequence of the characterisation of $F^{(\mu)}$ in (\ref{thm:haydon-e1}).
\end{proof}

Next we examine the preduals $F^{(\lambda)}$ as Banach spaces; while they give different weak$^*$-topologies from the canonical predual, it turns out that, purely as a Banach space, these preduals are all isomorphic to $c_0$. We begin with a pleasing form of the principle of local reflexivity which enables us to extend a finite sequence to an element of $F$ which behaves well outside the initial sequence.

\begin{lemma}\label{haydon:plr}
Let $y\in\ell_\infty(\Z)$ be a finitely supported sequence with support $I\subset\Z$ (that is, $I$ is a finite subset of $\Z$ and $y(t)=0$ for $t\not\in I$). Then there exists $x\in F$ with $x(t)=y(t)$ for $t\in I$ and $|x(t)|\leq\lambda^{-1}\|y\|_\infty$ for $t\not\in I$.
\end{lemma}
\begin{proof}
Since $F$ is shift-invariant, we can shift $y$ and assume that $I$ lies in some interval $[1,2^k]\cap\Z$ for some $k\in\mathbb N$.  Then define
\begin{equation}
x=\sum_{n=1}^{2^k}y(n)\sigma^n\tau^k(x_0),
\end{equation}
where $\sigma$ is the bilateral shift and $\tau$ the operator defined in (\ref{hadyon:DefT}).  Lemma~\ref{p:TInvariant} shows that $\tau^k(x_0)\in F$,
and so $x\in F$.  For $s,r\in\Z$ with $1 \leq r\leq 2^k$, we have
\begin{equation}
 x(2^ks+r) =\sum_{n=1}^{2^k}y(n)\tau^k(x_0)(2^ks+r-n).
\end{equation}
The terms in this sum are zero unless $2^ks+r-n$ is divisible by $2^k$, so only the $n=r$ term contributes.  Therefore $x(2^ks+r)=y(r)x_0(s)$.  As $x_0(0)=1$, we can take $s=0$ to obtain that $x(r)=y(r)$ for $1\leq r\leq 2^k$, so $x$ extends $y$.  When $s\neq 0$, we have
\begin{equation}
|x(2^ks+r)|=|y(r)||x_0(s)|\leq \lambda^{-1}\|y\|_\infty
\end{equation}
as $|x(s)|\leq\lambda^{-1}$ for $s\neq 0$.
\end{proof}

\begin{remark}\label{rem:3.7}
The previous lemma also shows that the preduals $F^{(\lambda)}$ are isometric preduals of $\ell_1(\Z)$, in that the canonical map $\iota_F$ is an isometry.  Indeed, given $a\in \ell_1(\Z)$, we estimate
\begin{equation}
\|\iota_F(a)\|=\sup_{\substack{y\in F\\ \|y\|\leq 1}}|\ip{a}{y}|\geq\sup_{\substack{x\in c_0(\Z)\\\|x\|\leq 1}}|\ip{a}{x}|=\|a\|_{\ell_1(\Z)}.
\end{equation}
where the inequality is established by choosing a finitely supported element $x$ which approximates the second supremum and using the previous lemma to produce a suitable $y$.  Since $\|\iota_F\|\leq 1$, it follows that $\iota_F$ is isometric.
\end{remark} 

Let $K$ be a compact Hausdorff space.  Recall that a closed subspace $X$ of $C(K)$ is called a \emph{$G$ space} if there is an index set $\Lambda$, and for each $\alpha\in\Lambda$, there are
$x_\alpha,y_\alpha\in K$ and $\lambda_\alpha$ such that $X=\{f\in C(K) :
f(x_\alpha) = \lambda_\alpha f(y_\alpha) \}$.  In \cite{ben}, Benyamini proved 
that every separable $G$ space is isomorphic to a space of the form $C(L)$ for some compact Hausdorff space $L$. As noted at the end of \cite{ben}, this result holds for both real and complex scalars.  The characterisation of our preduals $F^{(\lambda)}$ given in Theorem \ref{thm:haydon} show that these preduals are $G$-spaces, so Benyamini's result shows that each $F^{(\lambda)}$ is isomorphic, purely as a Banach space, to some $C(L)$ space.   To compute which space $L$ occurs, we shall use the Szlenk index, which classifies the isomorphism classes of $C(L)$ spaces.

The Szlenk index was introduced in \cite{szlenk}.  There are a number of equivalent
definitions of the Szlenk index, but we shall follow Rosenthal's survey article
\cite{ros}, as this also gives a self-contained treatment of the Szlenk index of
$C(K)$ spaces.  For a separable Banach space $E$ which contains no isomorphic copy
of $\ell^1$, it is shown in \cite[Proposition~2.17]{ros} that the definition we give
below, and Szlenk's original definition, give the same index.  Notice that if
$E$ is a predual of $\ell_1(\Z)$, then these conditions do apply to $E$.

Fix $\varepsilon>0$ and set $P_0(\varepsilon)=\{\mu\in E^* : \|\mu\|\leq 1\}$.
For a countable ordinal $\alpha$, supposing we have defined $P_{\beta}$ for
$\beta\leq\alpha$, we define $P_{\alpha+1}(\varepsilon)$ to be the weak$^*$-closure of
\begin{equation}
\Big\{ \mu\in P_\alpha(\varepsilon) : \exists(\mu_n)\subseteq
P_\alpha(\varepsilon) \text{ with } \mu_n\rightarrow\mu
\text{ weak$^*$, and } \|\mu_n-\mu\|\geq\varepsilon, \ n\in\mathbb N \Big\}.
\end{equation}
Note that here we only consider sequences $(\mu_n)$, and not nets.
If $\alpha$ is a limit ordinal, then we set $P_\alpha(\varepsilon)
= \bigcap_{\beta<\alpha} P_\beta(\varepsilon)$.  Then define
\begin{equation}
 \eta(\varepsilon,E) = \sup\{ \alpha : P_\alpha(\varepsilon)\not=\emptyset \} 
 \end{equation}
if this exists, or set $\eta(\varepsilon,E)=\omega_1$, the first uncountable ordinal, otherwise.  Finally, the Szlenk index of $E$ is defined as $\eta(E) = \sup_{\varepsilon>0} \eta(\varepsilon,E)$.  
 Let us note that if $\eta(\varepsilon,E)<\omega_1$, it must be a successor
ordinal. Indeed,
Êfrom the compactness of the $P_\alpha(\varepsilon)$, $\alpha< \eta(\varepsilon,E)
$, Êit follows that
Ê$\bigcap_{\alpha<\eta(\varepsilon,E)} P_\alpha(\varepsilon)\not=\emptyset$, and
thus $\eta(\varepsilon,E)$
cannot equal $\sup_{\alpha<\eta(\varepsilon,E)} \alpha$. On the other hand, $\eta(E) $ is always a limit ordinal,
  in fact, if $E^*$ is separable, then $\eta(E)=\omega^\alpha$  for some countable  ordinal $\alpha$
 \cite[Corollary 3.10]{ajo}.

The condition that $\ell_1$ does not embed into $E$ ensures that $\eta(E)<\omega_1$ if and only if $E^*$ is separable and so all our preduals have countable Szlenk index.

It is also common to define the Szlenk index without taking the weak$^*$-closure;
see \cite[Section~3]{lan} for example.
Bessaga and Pe{\l}czy{\'n}ski showed in \cite{BP} that if $K$ is an (infinite) countable
compact metric space, then $C(K)$ is isomorphic to $C(\omega^{\omega^\alpha}+1)$ for
some countable ordinal $\alpha\geq0$.  Furthermore, $C(\omega^{\omega^\alpha}+1)$
and $C(\omega^{\omega^\beta}+1)$ are isomorphic only when $\alpha=\beta$.
Then Samuel showed in \cite{sam} that $\eta( C(\omega^{\omega^\alpha}+1) ) = 
\omega^{\alpha+1}$.  In particular, we have that $c_0 \cong c = C(\omega^1+1)$ and
so $\eta(c_0) = \omega$.  A self-contained treatment of these results is given in
\cite[Section~2]{ros}.

\begin{theorem}\label{thm:szlenk_of_haydon}
For any $\lambda$, the Szlenk index of $F^{(\lambda)}$ is $\omega$,
and so $F^{(\lambda)}$ is isomorphic to $c_0$, as a Banach space.
\end{theorem}
\begin{proof}
Fix $\varepsilon>0$.  For $r>0$, denote by $\ell_1(\Z)_{[r]}$ the closed ball of radius $r$
in $\ell_1(\Z)$.  Suppose that $P_\alpha(\varepsilon) \subseteq \ell_1(\Z)_{[r]}$.  We will
show that $P_{\alpha+1}(\varepsilon) \subseteq \ell_1(\Z)_{[r']}$ where
\begin{equation} r' = r - \frac{\varepsilon}{3} \frac{1-|\lambda|^{-1}}{1+|\lambda|^{-1}}.
\label{Szlenk.e4} \end{equation}

We recall that Remark~\ref{rem:3.7} shows that $F^{(\lambda)}$ is an isometric
predual, and so we can use the $\ell_1$-norm on $(F^{(\lambda)})^* \cong \ell_1(\Z)$
when computing the Szlenk index.
We note that
\begin{align} P_{\alpha+1}(\varepsilon) &\subseteq P_{\alpha}(\varepsilon) \setminus
\bigcup\big\{ U : U\text{ is weak$^*$-open with } \operatorname{diam}(U\cap P_{\alpha}(\varepsilon)) < \varepsilon \big\} \\
&\subseteq \big\{ a\in P_\alpha(\varepsilon) : \exists\, (a_n)\subseteq P_\alpha(\varepsilon),
\ a_n\rightarrow a \text{ weak$^*$, and } \|a_n-a\|\geq\varepsilon/3 \big\} \label{Szlenk.e5}\\
&\subseteq \big\{ a\in \ell_1(\Z)_{[r]} : \exists\, (a_n)\subseteq \ell_1(\Z)_{[r]},
\ a_n\rightarrow a \text{ weak$^*$, and } \|a_n-a\|\geq\varepsilon/3 \big\}.
\end{align}
Here, for a subset $X$ of a normed space, $\operatorname{diam}(X) = \sup\{\|x-y\|:x,y\in X\}$.
It follows that if $x\in X$ and $\|x-y\|<\varepsilon/3$ for all $y\in X$, then
$\operatorname{diam}(X)\leq 2\varepsilon/3$, which shows the containment (\ref{Szlenk.e5}).

So, let $a\in \ell_1(\Z)_{[r]}$ and choose a sequence $(a^{(n)}) \subseteq \ell_1(\Z)_{[r]}$ converging weak$^*$ to $a$ (with respect to the topology induced by $F^{(\lambda)}$) and with $\| a-a^{(n)} \|\geq\varepsilon/3$ for all $n$.  By passing to a subsequence, we may suppose that for each $k\in\Z$, the scalar sequence $(a^{(n)}_k)$ converges, say to $b_k$.  Then
\begin{equation}
\|b\|_{\ell_1(\Z)} = \sum_{k\in\Z} |b_k| = \sum_{k\in\Z} \lim_{n\rightarrow\infty} |a^{(n)}_k|
\leq \sup_n \sum_{k\in\Z} |a^{(n)}_k|
= \sup_n \|a^{(n)}\| \leq r. 
\end{equation}

Let $\delta>0$ be much smaller than $\varepsilon$, and choose $N$ such that
$\sum_{|k|>N} |a_k|<\delta$ and $\sum_{|k|>N} |b_k|<\delta$.
Choose a norm one element $y\in\ell_\infty(\Z)$ such that $y(k)(a_k-b_k)=|a_k-b_k|$ for $|k|\leq N$ and with $y(k)=0$ when $|k|>N$. By Lemma \ref{haydon:plr}, there is some $x\in F^{(\lambda)}$ with $x(k)=y(k)$ for $|k|\leq N$ and $|x(k)|\leq\lambda^{-1}$ for $|k|>N$.  Then
\begin{equation}
\sum_{k\in\Z} x_k a_k = \ip{x}{a} = \lim_{n\rightarrow\infty} \ip{x}{a^{(n)}}
= \sum_{|k|\leq N} x_k b_k + \lim_{n\rightarrow\infty} \sum_{|k|>N} x_k a^{(n)}_k,
\end{equation}
and so
\begin{align}
\sum_{|k|\leq N} |a_k-b_k| - \sum_{|k|>N} |a_k| &\leq
\sum_{|k|\leq N} |a_k-b_k| - \Big| \sum_{|k|>N} a_kx_k \Big|
= \sum_{|k|\leq N} x_k(a_k-b_k) - \Big| \sum_{|k|>N} a_kx_k \Big|\nonumber \\
&\leq \Big| \sum_{k\in\Z} x_ka_k - \sum_{|k|\leq N} x_kb_k \Big|
= \lim_{n\rightarrow\infty} \Big| \sum_{|k|>N} x_k a_k^{(n)} \Big|\nonumber \\
&\leq|\lambda|^{-1}
  \liminf_{n\rightarrow\infty} \sum_{|k|> N} |a_k^{(n)}|.\label{Szlenk.e1}
\end{align}
Then
\begin{align} \varepsilon/3 &\leq \liminf_{n\rightarrow\infty} \| a^{(n)}-a \|
= \liminf_{n\rightarrow\infty} \sum_{k\in\Z} |a_k - a^{(n)}_k|
= \sum_{|k|\leq N} |a_k - b_k| + \liminf_{n\rightarrow\infty}  \sum_{|k|>N} |a_k-a^{(n)}_k| \nonumber\\
&\leq |\lambda|^{-1}\liminf_{n\rightarrow\infty}\sum_{|k|>N}|a_k^{(n)}|+\sum_{|k|>N}|a_k|+\delta+\liminf_{n\rightarrow\infty}\sum_{|k|>N}|a_k^{(n)}|\nonumber\\
&\leq2\delta+(1+|\lambda|^{-1})\liminf_{n\rightarrow\infty}\sum_{|k|>N}|a_k^{(n)}|.\label{Szlenk.e2}
\end{align}
Since each $a^{(n)}$ has $\ell_1$-norm at most $r$, we have that
\begin{equation}\label{Szlenk.e3}
\liminf_{n\rightarrow\infty}\sum_{|k|>N}|a_k^{(n)}|+\sum_{|k|\leq N}|b_k|=\liminf_{n\rightarrow\infty}\sum_{k\in\Z}|a_k^{(n)}|\leq r,
\end{equation}

Combining the estimates (\ref{Szlenk.e1}), (\ref{Szlenk.e2}) and (\ref{Szlenk.e3}) gives
\begin{align}
\|a\|_{\ell_1(\Z)} &\leq\sum_{|k|\leq N}|a_k-b_k|+\sum_{|k|\leq N}|b_k|+\sum_{|k|>N}|a_k|\nonumber\\
&\leq \Big(\delta+|\lambda|^{-1}\liminf_{n\rightarrow\infty}\sum_{|k|>N}|a_k^{(n)}|\Big)+\Big(r-\liminf_{n\rightarrow\infty}\sum_{|k|>N}|a_k^{(n)}|\Big)+\delta\nonumber\\
&=2\delta+r-(1-|\lambda|^{-1})\liminf_{n\rightarrow\infty}\sum_{|k|>N}|a_k^{(n)}|\nonumber\\
&\leq 2\delta+r-\frac{1-|\lambda|^{-1}}{1+|\lambda|^{-1}}\Big(\varepsilon/3-2\delta\Big).
\end{align}
Since $\delta>0$ was arbitrary, we have $\|a\|\leq r'$, where $r'$ is given by (\ref{Szlenk.e4}), as claimed.

By induction, we see that for any $\alpha\in\mathbb N$, we have
\begin{equation} P_\alpha(\varepsilon) \subseteq \Big\{ a\in\ell_1(\Z) : \|a\| \leq 1 - \alpha\frac{\varepsilon}{3}
\frac{|\lambda|-1}{|\lambda|+1} \Big\},
\end{equation}
and so $\eta(F^{(\lambda)},\varepsilon)$ is finite for all $\varepsilon>0$. Hence $\eta(F^{(\lambda)})=\omega$.  It then follows that $F^{(\lambda)}$ is isomorphic to $c_0$ by the discussion following Remark \ref{rem:3.7}.
\end{proof}

\begin{remark}
Note that the only property of the preduals $F^{(\lambda)}$ used in the proof of Theorem~\ref{thm:szlenk_of_haydon} is the strong form of the principle of local reflexivity obtained in Lemma~\ref{haydon:plr}.  We also used in the proof that $F^{(\lambda)}$ is an isometric
predual, but an easy modification would work for a merely isomorphic predual.  Thus
any predual $E$ satisfying the conclusion of Lemma~\ref{haydon:plr} (for some $|\lambda|>1$) has Szlenk index $\omega$.
\end{remark}

\section{Preduals and semigroup compactifications}\label{semigroup}

In this section we formulate a characterisation of shift-invariant preduals of $\ell_1(\Z)$ as submodules of the space $M(\S) = C(\S)^*$, where $\S$ is a \emph{semitopological semigroup compactification} of $\Z$. In the next section we will use this characterisation to produce more examples of shift-invariant preduals.
 
A \emph{semitopological semigroup} is a semigroup $(\S,+)$ endowed with a topology which renders addition separately continuous. If furthermore $\S$ is compact and $\Z$ can be densely embedded into $\S$, so that this embedding is a  semigroup homomorphism, we say that $\S$ is a \emph{semitopological semigroup compactification} of $\Z$.
 
Assume that $\S$ is  such a  semitopological semigroup compactification  of $\Z$. We consider $\Z$ to be a subset of $\S$. Since $\Z$ is dense in $\S$, $\S$ is an abelian semigroup.  The dual of the space of continuous function on $\S$, $C(\S)$, can be identified with the space $M(\S)$ of Borel measures on $\S$ with bounded variation, and $\ell_1(\Z)$ is in a canonical way  a subspace of  $M(\S)$.  The convolution on $\ell_1(\Z)$  extends to a convolution on $M(\S)$, i.e. for $\Phi,\Psi\in M(\S)$,
\begin{equation}\label{conv}
\la \Phi*\Psi, f\ra =\int f(s+t )\,d\Phi(s)\,d\Psi(t), \quad f\in C(\S).
\end{equation}
The fact that  $\la \Phi*\Psi, f\ra$ is well defined and that $\Phi*\Psi\in M(\S)$ is a consequence of 
\cite{joh}, the proof of which shows that $(s,t)\mapsto f(s+t)$ is measurable with respect to the product measure $\Phi\times\Psi$.  As such Fubini's
theorem allows us to interchange the order of integration in (\ref{conv})
and hence $*$ is commutative. In this way $M(\S)$ is an abelian Banach
algebra under convolution.  By restriction, we can regard $C(\S)$ as a space of bounded functions on $\Z$.  As $\Z$ is dense in $\S$, this identifies $C(\S)$ with a subspace of $\ell_\infty(\Z)$.  
  
We can now state our characterisation of shift-invariant preduals of $\ell_1(\Z)$ in terms of semigroup compactifications. At this stage we prove the first part of the theorem, showing that this construction induces shift-invariant preduals. We return to prove part 2 of the theorem, which demonstrates that every shift-invariant predual arises in this way, in Proposition \ref{thm:easy_way_topcmpt}.
    
\begin{theorem}\label{char-preduals}
\begin{enumerate}
\item Let $\S$ be a semitopological semigroup compactification  of $\Z$. Let $\Theta: M(S)\to \ell_1(\Z)$ be a bounded projection which is also a homomorphism with respect to convolution. Define 
\begin{equation}
F=\ ^\perp\ker\Theta=\{f\in C(\S): \la\Psi, f\ra=0,\ \text{for all }\Psi\in\ker\Theta\}.
\end{equation}
If $\ker\Theta$ is weak$^*$-closed, then $F$, identified as a subspace of $\ell_\infty(\Z)$, is a shift-invariant predual of $\ell_1(\Z)$.

\item Conversely, if $E\subseteq\ell_\infty(\Z)$ is a shift-invariant predual of $\ell_1(\Z)$ then there exists a semitopological semigroup compactification $\S$ of $\Z$, and a bounded projection $\Theta: M(\S)\to\ell_1(\Z)$, which is a  homomorphism with respect to convolution, such that $\ker\Theta$ is weak$^*$-closed in $M(\S)$, and such that $E=\ ^\perp\ker\Theta$. Moreover, $\S$ can be chosen  so that  the map $\S\to \ell_1(\Z)$, $s\mapsto \Theta(\delta_s)$, is injective. 
 \end{enumerate}
\end{theorem}

\begin{proof}[Proof of Theorem \ref{char-preduals}, part 1]
As $\Theta$ is a bounded homomorphism, $\ker\Theta$ is an ideal in $M(\S)$, and so $F={^\perp}\ker \Theta$ is a closed $\ell_1(\Z)$-submodule of $C(\S)$.  Let $E\subseteq\ell_\infty(\Z)$ be the image of $F$.  It follows that $E$ is shift-invariant. We need to show that $\iota_E:\ell_1(\Z)\rightarrow E^*$ is an isomorphism; by the Open Mapping Theorem, this is equivalent to showing that $\iota_E$ is bijective.

Let $a\in\ell_1(\Z)$ with $\iota_E(a)=0$.  Viewing $a$ as a member of $M(\S)$, it follows
that $\ip{a}{x}=0$ for all $x\in F$, so
$a \in ({^\perp}\ker \Theta)^\perp$.  As $\ker \Theta$ is weak$^*$-closed, it follows that
$\ker \Theta = ({^\perp}\ker \Theta)^\perp$, so $a\in\ker \Theta$.  But $\Theta(a)=a$, so $a=0$, and we conclude that $\iota_E$ is injective.  For surjectivity, take $\mu\in E^*$.  As $\Z$ is dense in $\S$, the restriction map $C(\S)\rightarrow
\ell_\infty(\Z)$ is an isometry, and hence the map $F\rightarrow E$ is also
an isometry, which induces $\tilde{\mu}\in F^*$ associated to $\mu$. Take a Hahn-Banach extension $\lambda\in C(\S)^* = M(\S)$  of $\tilde{\mu}$. As  $\lambda-\Theta(\lambda)\in\ker\Theta$, we have
\begin{equation}
\ip{\Theta (\lambda)}{x}=\ip{\lambda}{x} =\ip{\tilde\mu}{x}, \text{ for all $x\in F=\ {^\perp}\ker \Theta$}
\end{equation}
It follows that $\iota_E(\Theta(\lambda))=\mu$.
\end{proof}

In order to prove part 2 of Theorem \ref{char-preduals} in Proposition \ref{thm:easy_way_topcmpt} below and to associate semigroup compactifications to our shift-invariant preduals, we use weakly almost periodic functionals. While this theory is well developed in the abstract setting of Banach algebras and dual Banach algebras (see \cite{daws} for example) we only need it as it applies to $\ell_1(\Z)$, which we now review for the reader's convenience. An element $\mu\in\ell_\infty(\Z)$ is \emph{weakly almost periodic} if
the orbit of $\mu$ under the bilateral shift is a relatively weakly compact set. Alternatively one can use the Arens products $\aone$ and $\atwo$ on $\ell_\infty(\Z)^*\cong\ell_1(\Z)^{**}$ to specifiy the weakly almost periodic functionals.  Given a Banach algebra $A$, recall that $A^*$ has an $A$-module structure given by
\begin{equation}
 \ip{a\cdot\mu}{b} = \ip{\mu}{ba}, \quad \ip{\mu\cdot a}{b}
= \ip{\mu}{ab}, \qquad\mu\in A^*, a,b\in A.
\end{equation}
We can also define actions of $A^{**}$ on $A^*$ by
\begin{equation} \ip{\Psi \cdot\mu}{a} = \ip{\Psi}{\mu\cdot a}, \quad
\ip{\mu\cdot\Psi_1}{a} = \ip{\Psi_1}{a\cdot\mu},
\qquad a\in A, \mu\in A^*, \Psi\in A^{**}. 
\end{equation}
Finally, we define
\begin{equation} \ip{\Psi_1\aone\Psi_2}{\mu} = \ip{\Psi_1}{\Psi_2\cdot\mu}, \quad
\ip{\Psi_1\atwo\Psi_2}{\mu} = \ip{\Psi_2}{\mu\cdot\Psi_1},
\qquad \mu\in A^*, \Psi_1,\Psi_2\in A^{**}. 
\end{equation}
Then $\aone$ and $\atwo$ are associative, contractive products on $A^{**}$,
called the \emph{Arens products}.  The canonical map $\kappa_A:A\rightarrow A^{**}$ becomes a homomorphism for either Arens product.  These products can also be described via iterated limits. Given $\Psi_1,\Psi_2\in A^{**}$, take bounded nets $(a_{1,\alpha})$ and $(a_{2,\alpha})$ in $A$ converging weak$^*$ in $A^{**}$ to $\Psi_1$ and $\Psi_2$ respectively. The Arens products $\Psi_1\aone\Psi_2$ and $\Psi_1\atwo\Psi_2$ are then described by the following iterated limits (which are well defined):
\begin{equation}\label{eq:4.6}
\ip{\Psi_1\aone\Psi_2}{\mu} = \lim_\alpha\lim_\beta \ip{\mu}{a_{1,\alpha} a_{2,\beta}},
\quad \ip{\Psi_1\atwo\Psi_2}{\mu} =\lim_\beta \lim_\alpha \ip{\mu}{a_{1,\alpha} a_{2,\beta}},\quad \mu\in A^*.
\end{equation}
We now concentrate on the case that $A=\ell_1(\Z)$ with the convolution. 
The weakly almost periodic functionals are characterised as those $\mu\in \ell_\infty(\Z)$ for which $\ip{\Psi_1\aone\Psi_2}{\mu} =\ip{\Psi_1\atwo\Psi_2}{\mu}$ for all $\Psi_1,\Psi_2\in\ell_\infty(\Z)^*$.  This follows from the proof of Lemma 3.3 in \cite{gulick}.

Write $\wap(\Z)$ for the collection of these almost periodic elements of $\ell_\infty(\Z)$. The relevance of $\wap(\Z)$ to shift-invariant preduals is given by the next proposition.

\begin{proposition}
Let $F\subset\ell_\infty(\Z)$ be a concrete shift-invariant predual for $\ell_1(\Z)$. Then $F\subset\wap(\Z)$.
\end{proposition}
\begin{proof}
Given $\Psi_1,\Psi_2\in \ell_\infty(\Z)^*$, take bounded nets $(a_{1,\alpha})$ and $(a_{2,\beta})$ in $\ell_1(\Z)$ converging weak$^*$ in $\ell_\infty(\Z)^*$ to $\Psi_1$ and $\Psi_2$ respectively.
After passing to  subnets we can assume that $(a_{1,\alpha})$ and $(a_{2,\beta})$
 are weak$^*$-convergent to $a_1,a_2\in\ell_1(\Z)$ respectively with respect to the duality between $F$ and $\ell_1(\Z)$. For $\mu\in F$, it follows from the fact that the convolution multiplication in $\ell_1(\Z)$  is separately weak$^*$-continuous, that
\begin{align} \ip{\Psi_1\aone\Psi_2}{\mu}
&=\lim_\alpha  \lim_\beta\ip{\mu}{a_{1,\alpha}*a_{2,\beta}}
= \lim_\alpha \ip{\mu}{a_{1,\alpha}*a_{2}} = \ip{\mu}{a_1*a_2} \nonumber \\
&= \lim_\beta \ip{\mu}{a_{1}*a_{2,\beta}}
=\lim_\beta  \lim_\alpha\ip{\mu}{a_{1,\alpha}*a_{2,\beta}} = \ip{\Psi_1\atwo\Psi_2}{\mu}.
\end{align}
Thus $F\subset\wap(\Z)$.
\end{proof}

The descriptions above imply that $\wap(\Z)$ is closed under multiplication (in $\ell_\infty(\Z)$) and under 
 taking adjoints, and it is therefore  a C$^*$-subalgebra of $\ell_\infty(\Z)$, which is invariant under the bilateral shift and contains the unit $1$ of $\ell_\infty(\Z)$.  Write $\Z^\wap$ for the character space of $\wap(\Z)$, so that the Gelfand transform gives a canonical isometric isomorphism $\wap(\Z)\cong C(\Z^\wap)$. Each member of $\Z$ induces a character on $\Z^\wap$ by evaluation, and this gives us a map $\Z\rightarrow\Z^\wap$ which
has dense range.  Since $c_0(\Z)\subset\wap(\Z)$, this map is injective and so $\Z^\wap$ is a compactification of $\Z$. As we will review below, $\Z^\wap$ has a natural semigroup structure coming from the Arens products.  Furthermore, it is the maximal semigroup compactification of $\Z$, in the sense that given any other compact semitopological semigroup $\S$ and a homomorphism $\phi:\Z\rightarrow \S$, with dense range, then there exists a (necessarily unique)
continuous homomorphism $\tilde\phi:\Z^{\wap}\rightarrow \S$, such that the following
diagram is commutative:
\begin{equation}\label{wap:diag}
\xymatrix{ \Z \ar[r]^{\phi} \ar[d] & \S \\
\Z^{\wap} \ar[ru]_{\tilde\phi} }
\end{equation}

Let $F\subset\wap(\Z)$ be a closed, shift-invariant subspace.  Using the representation \eqref{eq:4.6} the Arens products can be used to show that the  product on $F^*=\ell_\infty(\Z)^*/F^\perp$ given by 
\begin{equation}
(\Psi_1+F^\perp)(\Psi_2+F^\perp)=(\Psi_1\aone\Psi_2)+F^\perp=(\Psi_1\atwo\Psi_2)+F^\perp,\quad \Psi_1,\Psi_2\in\ell_\infty(\Z)^*,
\end{equation}
is well defined and turns $F^*$ into a dual Banach algebra (see \cite[Proposition~2.4]{daws}, or, as $\ell_1(\Z)$ is commutative,
see \cite[Lemma~1.4]{lau}).
 Now consider a shift-invariant C$^*$-subalgebra $B$ of $\wap(\Z)$ containing $1$. Given $\mu,\nu\in B$ and $n\in\Z$, we have
\begin{equation}
 \ip{\mu\nu\cdot\delta_n}{\delta_m} = (\mu\nu)(n+m) = \mu(n+m) \nu(n+m)
= \ip{(\mu\cdot\delta_n)(\nu\cdot\delta_n)}{\delta_m}, \quad m\in\Z,
\end{equation}
so that $(\mu\nu)\cdot\delta_n=(\mu\cdot\delta_n)(\nu\cdot\delta_n)$. 
Thus, for a character $\Psi$ on $B$,
\begin{equation} \ip{\Psi\cdot \mu\nu}{\delta_n} = \ip{\Psi}{\mu\nu\cdot\delta_n}
= \ip{\Psi}{\mu\cdot\delta_n}\ip{\Psi}{\nu\cdot\delta_n}
= \ip{(\Psi\cdot \mu)(\Psi\cdot \nu)}{\delta_n},\quad n\in\Z,
\end{equation}
so that $\Psi\cdot(\mu\nu)=(\Psi\cdot\mu)(\Psi\cdot\nu)$. Therefore, for
characters $\Psi_1$ and $\Psi_2$ on $B$,
\begin{equation}
\ip{\Psi_1\aone\Psi_2}{\mu\nu} = \ip{\Psi_1}{\Psi_2\cdot \mu\nu}
= \ip{\Psi_1}{\Psi_2\cdot \mu} \ip{\Psi_1}{\Psi_2\cdot \nu}
= \ip{\Psi_1\aone\Psi_2}{\mu} \ip{\Psi_1\aone\Psi_2}{\nu},
\end{equation}
and so $\Psi_1\aone\Psi_2$ is also a character on $B$.  Let $\hat B$ be the character space of
$B$, so that the product on $B^*$ restricts to a product on $\hat B$.  Since the product on $B^*$ is separately weak$^*$-continuous, this turns $\hat B$ into a compact semitopological semigroup.  Furthermore, for each $n\in\Z$, evaluation at $n$ gives a character $\delta_n$, and since $\aone$ extends the product on $\ell_1(\Z)$, this gives
a semigroup homomorphism from $\Z$ to $\hat B$.

In particular, we can apply the previous paragraph when $B=\wap(\Z)$, and so $\Z^\wap$ becomes
a compact semitopological semigroup, and the two Arens products on $M(\Z^\wap)$ coincide with the convolution, introduced in (\ref{conv}).
Now take another compact semitopological semigroup $\S$ and a homomorphism $\phi:\Z\rightarrow \S$.  This induces a $*$-homomorphism $\theta:C(\S) \rightarrow \ell_\infty(\Z)$. As $\S$ is semitopological and compact, it is easily checked that $\theta(C(\S))\subset\wap(\Z)$, and so $\theta$ induces the continuous map $\tilde\phi:\Z^\wap \rightarrow \S$ so that the diagram (\ref{wap:diag}) commutes.  The density of $\Z$ in $\Z^\wap$ ensures that $\tilde\phi$ is a semigroup homomorphism and is uniquely determined.  By replacing $\S$ by the closure of $\phi(\Z)$ in $\S$ we may always assume that $\phi(\Z)$ is dense in $\S$, in which case $\theta:C(\S)\rightarrow\wap(\Z)$ will be injective, and hence an isometry onto its range.

Given  a semitopological semigroup  compactification of  $\S$, and let $\theta: C(\S)\to \wap(\Z)$ be defined as above. Then
$\theta^*:\wap(\Z)^*\rightarrow M(\S)$ is a homomorphism 
 with respect to convolution. To see this, it suffices to check that $\theta^*(\delta_{n+m})=\theta^*(\delta_n)\theta^*(\delta_m)$ for $m,n\in\Z$ as $\ell_1(\Z)$ is weak$^*$-dense in the dual Banach algebra $\wap(\Z)^*$. This follows as $\theta^*(\delta_n)=\delta_{\phi(n)}$, and so
\begin{align} \ip{\theta^*(\delta_n)* \theta^*(\delta_m)}{x}
&= \ip{\delta_{\phi(n)}* \delta_{\phi(m)}}{x}
= \int_{\S\times \S} x(s+t) \ d\delta_{\phi(n)}(s) \ d\delta_{\phi(m)}(t)
= x\big(\phi(n) + \phi(m)\big) \nonumber\\ &= x\big(\phi(n+m)\big)
= \ip{\theta(x)}{\delta_{n+m}} = \ip{\theta^*(\delta_{n+m})}{x},\quad x\in C(\S).\end{align}

Now suppose that $E\subseteq\ell_\infty(\Z)$ is a shift-invariant predual for
$\ell_1(\Z)$, and let $B$ be the unital C$^*$-algebra generated by $E$ in
$\ell_\infty(Z)$.  As $E$ is shift-invariant, it follows that $B$ is also, and
as $E\subseteq\wap(\Z)$, also $B\subseteq\wap(\Z)$.  Thus $B=C(\hat{B})$ for
some compact semitopological semigroup $\hat{B}$.  We have the commutative
diagram
\begin{equation}\label{wap:diag2} \xymatrix{ E^* & B^* \ar@{->>}[l]_-q & \ell_\infty(\Z)^* \ar@{->>}[l] \\
& \ell_1(\Z) \ar[lu]^{\iota_E} \ar[u]^{\iota_B} \ar[ru]_{\kappa_{\ell_1(\Z)}} } \end{equation}
where the maps along the top are quotients.  As $\iota_E$ is an isomorphism,
it follows that $\iota_B:\ell_1(\Z)\rightarrow B^*$ is an isomorphism
onto its range.  Now, $B^*=M(\hat{B})$ which is a dual Banach algebra equipped with the product from (\ref{conv}), and $\iota_B$ is an algebra homomorphism. Note too that the homomorphism $\Z\rightarrow\hat{B}$ is injective. This follows, as $E$, and hence $B$, separates the points of $\ell_1(\Z)$.
Indeed, if we denote $\phi$ the map $\Z\rightarrow\hat B$, then
$\iota_B(\delta_n) = \delta_{\phi(n)} \in M(\hat B)=B^*$.
We are now finally in a position to associate a semigroup and homomorphic projection to a shift-invariant predual, and to prove the second part of Theorem \ref{char-preduals}.

\begin{proposition}\label{thm:easy_way_topcmpt}
Let $E\subseteq\ell_\infty(\Z)$ be a shift-invariant predual for $\ell_1(\Z)$,
and form $B=C^*(1,E)$ as above.  There is a bounded Banach algebra homomorphism $\Theta:M(\hat{B})\rightarrow\ell_1(\Z)$ such that $\iota_B\Theta$ is a projection on $M(\hat{B})$.  Furthermore, $\ker \Theta$ is weak$^*$-closed, and
\begin{equation}
 E = {^\perp}\ker \Theta = \big\{ x\in B : \ip{\Psi}{x}=0,\ \Psi\in B^*,\ \Theta(\Psi)=0 \big\}.
 \end{equation}
The map $\hat{B}\rightarrow\ell_1(\Z)$ given by $\gamma\mapsto\Theta(\delta_\gamma)$ is injective.
\end{proposition}
\begin{proof}
We define a bounded linear map $\Theta=\iota_E^{-1} q:B^*\rightarrow\ell_1(\Z)$, where $q$ is the quotient map $B^*\rightarrow E^*=B^*/E^\perp$.  The commutative diagram in (\ref{wap:diag2}) shows that $\Theta\iota_B=\id_{\ell_1(\Z)}$ and so $\iota_B\Theta$ is a projection onto $\iota_B(\ell_1(\Z))$. By construction, $\ker\Theta=E^\perp$ which is weak$^*$-closed in $B^*$ and so
$E = {^\perp}(E^\perp) = {^\perp}\ker \Theta$.

We now check that $\Theta$ is an algebra homomorphism.  Given $x\in E$ and $a\in\ell_1(\Z)$, we have $x\cdot a\in E$, as $E$ is shift-invariant, and hence is an $\ell_1(\Z)$-module.  For $\Psi\in B^*$ and $x\in E$,
\begin{align}
 \ip{\Psi\cdot x}{a} &= \ip{\Psi}{x\cdot a} = \ip{q(\Psi)}{x\cdot a}
= \ip{x\cdot a}{\iota_E^{-1}q(\Psi)}\nonumber \\
&= \ip{x\cdot a}{\Theta(\Psi)} = \ip{\Theta(\Psi)\cdot x}{a}.
 \end{align}
It follows that $\Psi\cdot x = \Theta(\Psi)\cdot x$.  Similarly, 
$x\cdot\Psi = x\cdot \Theta(\Psi)$.  Thus, for $\Psi_1,\Psi_2\in B^*$ and $x\in E$,
\begin{equation} \ip{x}{\Theta(\Psi_1*\Psi_2)} = \ip{\Psi_1*\Psi_2}{x} =
\ip{\Psi_1}{\Psi_2\cdot x} = \ip{\Psi_1}{\Theta(\Psi_2)\cdot x} =
\ip{x}{\Theta(\Psi_1)*\Theta(\Psi_2)},
\end{equation}
showing that $\Theta$ is a homomorphism.

Finally, suppose that $\gamma_1,\gamma_2\in\hat B$ are distinct, and such that $\Theta(\delta_{\gamma_1})=\Theta(\delta_{\gamma_2})$.  Thus $\ip{\gamma_1}{x}
= \ip{\gamma_2}{x}$ for $x\in E$.  As a subspace of $C(\hat B)$, this means that
$E$ fails to separate the points $\gamma_1$ and $\gamma_2$.  As $C(\hat B)$ is
generated by $1$ and $E$, it follows that $C(\hat B)$ does not separate the points
$\gamma_1$ and $\gamma_2$, which is a contradiction.
So $\hat B\rightarrow \ell_1(\Z), \gamma\mapsto\Theta(\delta_\gamma)$ is injective.
\end{proof}

Given a shift-invariant predual $E\subset\ell_\infty(\Z)$ we say that $(\S,\Theta)$ \emph{induces} $E$ if $\S$ and $\Theta$ satisfy the hypotheses of part 1 of Theorem \ref{char-preduals} giving $E$ as the resulting predual.  In particular, given any predual $E$, Proposition \ref{thm:easy_way_topcmpt} gives a pair $(\hat{B},\Theta)$ inducing $E$.  The next section will focus on examples of preduals produced by Theorem \ref{char-preduals}; the rest of this section investigates the general theory which arises from constructions of this type.  First we note how to compute weak$^*$-limits in $\ell_1(\Z)$ with respect to these preduals. This approach is well adapted to finding the limit points of the set $\{\delta_n:n\in\Z\}$.  In the next proposition all weak$^*$-limits in $M(\S)$ are computed with respect to $C(\S)$, while weak$^*$-limits in $\ell_1(\Z)$ are with respect to $E$.

\begin{proposition}\label{convergence}
Let $(\S,\Theta)$ induce the shift-invariant predual $E\subset\ell_\infty(\Z)$. 
\begin{enumerate}
\item Let $(a_\alpha)$ be a bounded net in $\ell_1(\Z)$ converging weak$^*$ to $\mu\in M(\S)$. Then $(a_\alpha)$ converges weak$^*$ to $\Theta(\mu)$ in $\ell_1(\Z)$.
\item Suppose $(\gamma_k)$ is a net in $\S$ converging to $\gamma$. Then $\Theta(\delta_{\gamma_k})\rightarrow\Theta(\delta_\gamma)$ weak$^*$ in $\ell_1(\Z)$.
\item Given any subset $\S_0$ of $\S$, the weak$^*$-closure  of $\{\Theta(\delta_{\gamma_0}):\gamma_0\in\S_0\}$ in $\ell_1(\Z)$ is $\{\Theta(\delta_\gamma):\gamma\in\overline{\S_0}\}$.
\end{enumerate}
\end{proposition}
\begin{proof}
1.~~We have that $\ip{\mu}{x} = \lim_\alpha \ip{a_\alpha}{x}$ for $x\in F$.
As $F = {^\perp}\ker \Theta$, we see that $\ip{\mu}{x} = \ip{x}{\Theta(\mu)}$ for $x\in F$.  It follows that $a_\alpha\rightarrow \Theta(\mu)$ weak$^*$ with respect to $E$.

2.~~Suppose that $\gamma_k\rightarrow\gamma$ in $\S$, so that $\delta_{\gamma_k}\rightarrow
\delta_\gamma$ weak$^*$ in $M(\S)$.
  Observe that $\delta_{\gamma_k}-\Theta(\delta_{\gamma_k}) \in \ker{\Theta}$ for each $k$.  Pick some
subnet of $(\gamma_k)$, and then pass to a further subnet $(\gamma_j)$ to ensure that
$\Theta(\delta_{\gamma_j})$ converges weak$^*$ to $\mu\in M(\S)$, so that
$\delta_{\gamma_j} - \Theta(\delta_{\gamma_j}) \rightarrow \delta_\gamma - \mu$
weak$^*$ in $M(\S)$.   As $\ker\Theta$
is weak$^*$-closed, it follows that $\delta_\gamma - \mu \in \ker\Theta$, that is,
$\Theta(\mu) = \Theta(\delta_\gamma)$.  By part~1, it follows that
$\Theta(\delta_{\gamma_j}) \rightarrow \Theta(\delta_\gamma)$ weak$^*$ in $\ell_1(\Z)$.
As every subnet of $\Theta(\delta_{\gamma_k})$ has a subnet converging to
$\Theta(\delta_\gamma)$, the statement follows.

3.~~Given a net $(\gamma_k)$ such that $\Theta(\delta_{\gamma_k})$ is weak$^*$-convergent in $\ell_1(\Z)$ we can pass to a subnet so that $\gamma_k\rightarrow\gamma$ in $\S$, whence the result follows from the previous part.
\end{proof}

A pair $(\S,\Theta)$ used to construct a shift-predual $E$ via the first part of 
Theorem~\ref{char-preduals} may have an unnecessarily large semigroup. To this end we say that a pair $(\S,\Theta)$ inducing a predual $E$ is \emph{minimal} if the semigroup homomorphism $\S\rightarrow\ell_1(\Z)$ given by $\gamma\mapsto\Theta(\delta_\gamma)$ is injective.  Of course, the pair $(\hat B,\Theta)$ constructed by Proposition~\ref{thm:easy_way_topcmpt} is minimal.  Clearly, if we start with $E$, and form
$(\hat B,\Theta)$, then $E$ can be reconstructed by part 1 of Theorem~\ref{char-preduals}. 
The next few results show that a minimal pair is uniquely determined by the predual and examine restrictions on the structure of the semigroup in a minimal pair.

\begin{lemma}\label{minimal_is_canonical}
Let $(\S,\Theta)$ be minimal, construct $E$ using part 1 of Theorem~\ref{char-preduals},
and then use Proposition~\ref{thm:easy_way_topcmpt} to construct $(\hat B,\Theta')$ say.
Then $\hat B$ is canonically isomorphic to $\S$, and under this identification,
$\Theta$ and $\Theta'$ agree.
\end{lemma}
\begin{proof}
Using the notation of  Proposition \ref{thm:easy_way_topcmpt}, we claim that $C^*(F,1) = C(\S)$.
This will follow if we can show that $F$ separates the points of $\S$.  Indeed, suppose that $\gamma_1,\gamma_2\in\S$ satisfy $f(\gamma_1)=f(\gamma_2)$ for each $f\in F$.  Then $\ip{\delta_{\gamma_1}-\delta_{\gamma_2}}{f} = 0$ for each $f\in F={}^\perp\ker\Theta$,
so $\delta_{\gamma_1}-\delta_{\gamma_2} \in \ker\Theta$, as $\ker\Theta$ is weak$^*$-closed.
Thus $\Theta(\delta_{\gamma_1}) = \Theta(\delta_{\gamma_2})$, so by minimality,
$\gamma_1=\gamma_2$, as required.  We shall henceforth identify $\hat B$ with $\S$.

We shall be careful with identifications.  We regard $F$ as a subspace of $C(\S)$,
and by restriction of functions on $\S$ to functions on $\Z$, we obtain $E$.
Let $r:F\rightarrow E \subseteq \ell_\infty(\Z)$ be this restriction map, and let $j:\ell_1(\Z)\rightarrow M(\S)$ be the inclusion, so that
$\ip{j(a)}{f} = \ip{r(f)}{a}$ for $a\in\ell_1(\Z)$ and $f\in F$.  As $\Z$ is dense
in $\S$, the map $r$ is an isomorphism, and so also $r^*:E^*\rightarrow F^*$ is an
isomorphism.  Let $q:C(\S)^*=M(\S)\rightarrow F^*$ be the quotient map, and
recall the map $\iota_E^{-1}:E^*\rightarrow\ell_1(\Z)$.  Then $\Theta' = 
\iota_E^{-1} (r^*)^{-1} q$.  As
\begin{equation}
\ip{\iota_E(a)}{r(f)} = \ip{r(f)}{a} = \ip{j(a)}{f} = \ip{qj(a)}{f},
\qquad a\in\ell_1(\Z), f\in F, 
\end{equation}
it follows that $qj = r^*\iota_E$, and so $qj\Theta' = r^* \iota_E \Theta' = q$.  As
$F^\perp = \ker\Theta$, for $\mu\in M(\S)$, we
have $j\Theta'(\mu) - \mu \in \ker\Theta$, that is, $\Theta(\mu) = \Theta(j\Theta'(\mu))
= j\Theta'(\mu)$, as $\Theta$ is a projection onto $j(\ell_1(\Z))$.  Thus
$\Theta = \Theta'$ under the appropriate identifications.
\end{proof}

\begin{remark}\label{rem:4.7}
Let $(\S,\Theta)$ be a minimal pair inducing $E$.  As $E^*\cong\ell_1(\Z)$, it follows that
$E$ is separable, and so also $B=C^*(E,1)$ is separable.  Then the closed unit ball of
$B^*$ is metrisable, and hence $\hat B$ is metrisable.  In particular, in the minimal case it is enough to
consider only sequences to understand the topology of $\hat B = \S$.
\end{remark}

When a semigroup compactification $\S$ of $\Z$ is countable, a standard Baire category argument shows that the points of $\Z$ are isolated in $\S$, and so in this case the embedding $\Z\rightarrow\S$ is a homeomorphism onto its range. On the other hand $\Z\rightarrow \mathbb T;\ n\mapsto e^{in}$ is a (semi)group homomorphism with dense range in which the points of $\Z$ are not isolated in their image. We have not been able to determine whether the semigroup $\S$ in a minimal pair $(\S,\Theta)$ inducing a
shift-invariant predual is necessarily countable; nevertheless the next proposition shows that points of $\Z$ are always isolated in $\S$.

\begin{proposition}\label{s-app1}
Let $E\subset\ell_\infty(\Z)$ be a shift-invariant predual for $\ell_1(\Z)$.
\begin{enumerate}
\item $\lambda\delta_0$ is not a weak$^*$-limit point of the set $\{\delta_n:n\in\Z\}$, for
any $\lambda\in\mathbb T$.
\item Let the pair $(\S,\Theta)$ induce $E$.
Then $\{0\}$ is open in $\S$, and so in particular, the homomorphism $\Z\rightarrow\S$ is a homeomorphism onto its range.
\end{enumerate}
\end{proposition}
\begin{proof}
For 1, we use the Szlenk index.  As $E$ is separable, the weak$^*$-topology on bounded
subsets of $\ell_1(\Z)$ is metrisable, and so we may work with sequences.
Suppose that some sequence $(\delta_{k_m})_{m=1}^\infty$ converges weak$^*$ to $\lambda\delta_0$ with respect to $E$.  Using the notation of Section~\ref{haydon}, certainly $\delta_n \in P_0(\varepsilon)$ for $n\in\Z$ and any $\varepsilon>0$.  Notice also that each $P_\alpha(\varepsilon)$ is
invariant under multiplying by any element of $\mathbb T$.
Suppose that $\{ \delta_n : n\in\Z \} \subseteq P_\alpha(\varepsilon)$
for an ordinal $\alpha$ and $0<\varepsilon<2$.  Then, as $\lim_m \delta_{k_m+n}
= \lambda \delta_n$ weak$^*$, and $\liminf_m \| \delta_{k_m+n} - \lambda\delta_n\|=2$, it follows that $\lambda\delta_n$, and hence also $\delta_n$, is a member of 
$P_{\alpha+1}(\varepsilon)$, for any $n\in\Z$.  However, then
$\delta_0\in P_\beta(\varepsilon)$ for any $\beta$ and $0<\varepsilon<2$, which
contradicts the countability of the Szlenk index of $E$.

For 2, we show that $\Z\rightarrow\S$ is a homeomorphism onto its range.  To do so,
we need to show that if $n\in\Z$ and $(n_\alpha)$ is a net in $\Z$ with
$n_\alpha\rightarrow n$ in $\S$, then $n_\alpha\rightarrow n$ in $\Z$, that is,
$n_\alpha=n$ eventually.  By part~2 of
Proposition~\ref{convergence}, it follows that $\delta_{n_\alpha} =
\Theta(\delta_{n_\alpha}) \rightarrow \Theta(\delta_n) = \delta_n$ weak$^*$ in $\ell_1(\Z)$.
Thus we see that $\{ \delta_{n_\alpha-n} \}$ has $\delta_0$ as a limit point, which by the first part, can only occur if, eventually, $\delta_{n_\alpha-n}
= \delta_0$, that is, $n_\alpha=n$.  As $\S$ is Hausdorff, it follows
immediately that $\{0\}$ is open in $\S$.
\end{proof}

\begin{lemma}\label{lem:4.8}
Let $(\S,\Theta)$ be a minimal pair inducing a shift-invariant predual $E$. Then $\S$ has exactly two idempotents, $0\in\Z$ and $\infty$.  The idempotent $\infty$ is a semigroup zero, i.e. $\infty+\gamma=\infty$ for all $\gamma\in\Z$ and, given any $\gamma\neq 0$ in $\S$, $\infty$ is a limit point of the set $\{\gamma n:n\in\mathbb N\}$.\end{lemma}
\begin{proof}
Certainly $0\in\S$ is idempotent.  By minimality, $\S$ embeds as a subsemigroup of $\ell_1(\Z)$ which has exactly two idempotents $\delta_0$ and $0_{\ell_1(\Z)}$ (to see this, take the Fourier transform into $C(\mathbb T)$).  Thus $\S$ has at most two idempotents. Take $\gamma\neq 0$ in $\S$. The closure $\overline{\{n\gamma :n\in\mathbb N}\}$ is a compact Hausdorff semitopological semigroup, and thus contains an idempotent, say $\gamma_0$,
see for example \cite[Chapter~1, Theorem~3.11]{semi}.

Suppose, towards a contradiction, that $\gamma_0=0$.  By Proposition~\ref{s-app1},
$\{0\}$ is open in $\S$, and so in particular, we can find $m>0$ with $m\gamma=0$.
Thus $\Theta(\delta_\gamma)^m = \delta_0$, and so applying the Fourier transform, we
see that $\Theta(\delta_\gamma) = \lambda \delta_0$ where $\lambda\in\mathbb T$ with
$\lambda^m=1$.  We can find a sequence $(n_k)$ in $\Z$ with $n_k\rightarrow\gamma$,
and so $\delta_{n_k} \rightarrow \lambda\delta_0$ weak$^*$, which contradicts
Proposition~\ref{s-app1}.  Thus $\gamma_0\not=0$.

We conclude that $\S$ has exactly two idempotents: $0$ and $\infty$ say.  Furthermore,
we have just shown that for any $0\not=\gamma\in\S$, the closure of
$\{n\gamma:n\in\mathbb N\}$ contains $\infty$.  Given any $\gamma\in\S$, 
\begin{equation}
\Theta(\delta_{\gamma+\infty})=\Theta(\delta_{\gamma})*\Theta(\delta_\infty)=\Theta(\delta_\gamma)*0=0=\Theta(\delta_\infty),
\end{equation}
so by injectivity of the map $\S\rightarrow\ell_1(\Z)$, we have that $\gamma+\infty=\infty$.
\end{proof}

Recalling that $\Z^\wap$ has infinitely many (indeed, $2^{2^\omega}$ many) idempotents (see, for example, \cite[Corollary~4.13]{rup}), it follows that $\S$ certainly cannot be all of $\Z^\wap$, if it satisfies the conclusions of Lemma \ref{lem:4.8}.

The Szlenk index defined in Section~\ref{haydon} provides a tool enabling us to better understand the possible Banach space isomorphism classes of our preduals.  Let $E\subset\ell_\infty(\Z)$ be a shift-invariant predual and let $(\S,\Theta)$ be a pair inducing $E$.  Since $E\subseteq C(\S)$, the Szlenk index of $E$ is at most that of
$C(\S)$; when $\S$ is countable, this can be computed using the Cantor-Bendixson
index (see \cite{ros}).  To find a lower bound for the Szlenk index of $E$, we proceed
as follows.
For $\varepsilon>0$, we define sets $\S_\alpha(\varepsilon)$ corresponding to ordinals $\alpha$ as follows. Set $\S_0(\varepsilon)=\S$. Given $\S_\alpha(\varepsilon)$, define
$$
\S_{\alpha+1}(\varepsilon)=\{\gamma\in \S_\alpha(\varepsilon):\exists\text{ a sequence }(\gamma_k)\text{ in }\S_{\alpha}(\varepsilon)\text{ converging to }\gamma\text{ with }\|\Theta(\delta_{\gamma_k})-\Theta(\delta_\gamma)\|\geq\varepsilon\}.
$$
For a limit ordinal $\beta$, set $\S_\beta(\varepsilon)=\bigcap_{\alpha<\beta}\S_\alpha(\varepsilon)$.

\begin{lemma}\label{lem:tildeS}
Let $(\S,\Theta)$ be a pair inducing a shift-invariant predual $E$, and form
$\S_\alpha(\varepsilon)$ as above.  Let $K\geq1$ be such that $K\|a\|_{\ell_1} \geq
\|a\|_{E^*} \geq K^{-1}\|a\|_{\ell_1}$ for each $a\in\ell_1(\Z)$.
For each ordinal $\alpha$ and $\varepsilon>0$, let
\begin{equation} \tilde\S_\alpha(\varepsilon) =
\{ K^{-1}\|\Theta\|^{-1} \Theta(\delta_\gamma) : \gamma\in\S_\alpha(\varepsilon) \}.
\end{equation}
Then
\begin{equation}
\tilde\S_\alpha(\varepsilon) \subseteq P_\alpha(\varepsilon \|\Theta\|^{-1} K^{-2}).
\label{eq:two} \end{equation}
Thus $\sup_{\varepsilon>0} \sup\{\alpha:\S_\alpha(\varepsilon)\not=\emptyset\}$
is at most the Szlenk index of $E$, and in particular is countable.
\end{lemma}
\begin{proof}
Let $c = \|\Theta\|^{-1} K^{-1}$ and $\varepsilon' = K^{-2} \|\Theta\|^{-1} \varepsilon$
For $\gamma\in\S$, we see that
\begin{equation} c \| \Theta(\delta_\gamma) \|_{E^*} \leq
\|\Theta\|^{-1} \| \Theta(\delta_\gamma) \|_{\ell_1} \leq 1 \end{equation}
It follows that 
$\tilde\S_0(\varepsilon) \subseteq P_0(c\varepsilon') = \{ a\in\ell_1(\Z) : \|a\|_{E^*}\leq1\}$.
Suppose now that $\tilde\S_\alpha(\varepsilon) \subseteq P_\alpha(c\varepsilon)$.
Let $\gamma\in\S_{\alpha+1}(\varepsilon)$, so $\gamma\in\S_\alpha(\varepsilon)$ and there
exists a sequence $(\gamma_k)$ in $\S_\alpha(\varepsilon)$ with $\gamma_k\rightarrow\gamma$,
and with $\|\Theta(\delta_{\gamma_k}) - \Theta(\delta_\gamma)\|_{\ell_1}\geq\varepsilon$ for
each $k$.  Let $a = c \Theta(\delta_\gamma)$ and $a_k = c\Theta(\delta_{\gamma_k})$ for
each $k$.  By part~2 of Proposition~\ref{convergence}
we have that $a_k \rightarrow a$ weak$^*$, and by assumption,
$a\in P_\alpha(\varepsilon')$ and $(a_k) \subseteq P_\alpha(\varepsilon')$.
Then observe that $\|a_k-a\|_{E^*} \geq K^{-1} \|a_k-a\|_{\ell_1}
= K^{-1} c \|\Theta(\delta_{\gamma_k}) - \Theta(\delta_\gamma)\|_{\ell_1}
\geq K^{-2} \|\Theta\|^{-1} \varepsilon \geq \varepsilon'$ for each $k$, from which it
follows that $a\in P_{\alpha+1}(\varepsilon')$.  Thus $\tilde\S_{\alpha+1}(\varepsilon)
\subseteq P_{\alpha+1}(\varepsilon')$.
\end{proof}

This gives us a criterion for exhibiting a shift-invariant predual which is not isomorphic as a Banach space to $c_0$.  Examples of this phenomena will be given in the next section.  Note too that minimality of the pair $(\S,\Theta)$ was not used in the calculations above; though if $(\S,\Theta)$ is not minimal, then the condition that $\|\Theta(\delta_{\gamma_k})-\Theta(\delta_\gamma)\|\geq\varepsilon$ is more restrictive.

\begin{proposition}\label{notco}
Let $E\subset\ell_\infty(\Z)$ be a shift-invariant predual for $\ell_1(\Z)$. Suppose that $a\in\ell_1(\Z)$  is a weak$^*$-accumulation point of the point masses $\{\delta_t:t\in\Z\}$ and has $\|a^n\|\geq 1$ for all $n\in\mathbb N$. Then $E$ is not isomorphic to $c_0$ as a Banach space.
\end{proposition}
\begin{proof}
Let $(\S,\Theta)$ be a minimal pair inducing $E$.  By Proposition~\ref{convergence} part~3,
we know that $a = \Theta(\delta_\gamma)$ for some $\gamma\in\S\setminus\Z$.  Given $0<\varepsilon<1$,
we claim that $\infty\in\S_\alpha(\varepsilon)$ for all finite $\alpha$.  It will then follow that 
$\infty\in\S_{\omega}(\varepsilon)$.  By the proof of Lemma \ref{lem:tildeS}
 this implies that $P_\omega(\varepsilon)\not=\emptyset$ for small enough $\varepsilon>0$. As noted in section \ref{haydon}, $\eta(\varepsilon)=\sup\{\alpha:P_\alpha(\varepsilon)\neq\emptyset\}$ is not a limit ordinal and so $\eta(\varepsilon)$ is strictly bigger than $\omega$ for all sufficiently small $\varepsilon>0$. Thus the Szlenk index of $E$ is strictly bigger than $\omega$, and such $E$ cannot be isomorphic to $c_0$.

In order to show that  $\infty\in\S_\alpha(\varepsilon)$ for all finite $\alpha$, we prove 
 by induction  for all $n\in\Z_+$,  that 
\begin{equation}
\{m\gamma+t:m\geq n,\ t\in\Z\}\cup\{\infty\}\subset S_n(\varepsilon),
\end{equation}
a hypothesis that is trivially satisfied when $n=0$.  By Remark \ref{rem:4.7}, we can find a sequence $(t_i)\subseteq\Z$
with $t_i\rightarrow\gamma$ in $\S$.  It follows that $m\gamma+t+t_i\rightarrow (m+1)\gamma+t$, while $\liminf\|a^m\delta_{t+t_i}-a^{m+1}\delta_t\|\geq2>\varepsilon$ as we must have $|t_i|\rightarrow\infty$ so the support of $a^m\delta_{t+t_i}$ is eventually shifted away from the support of $a^{m+1}\delta_t$.  Thus $(m+1)\gamma+t\in\S_{n+1}(\varepsilon)$.  Since $\infty$ is a limit point of $\{m\gamma:m\geq n\}$ and $\|a^m\|\geq1$ for all $m$, we have $\infty\in S_{n+1}(\varepsilon)$, establishing the claim.
\end{proof}

\begin{remark}
Let $G$ be a discrete group, and form the Banach space $\ell_1(G)$.  This becomes
a Banach algebra for the convolution product.  Then $\ell_\infty(G)$ becomes an
$\ell_1(G)$-bimodule, and this allows us to make sense of a predual $E\subseteq
\ell_\infty(G)$ being shift-invariant.  Again, this corresponds to $E$ turning
$\ell_1(G)$ into a dual Banach algebra.  Most of the results of this section
hold in this more general setting (in particular, $\wap(G)$ is a well-understood
object) with the exception of the final few results, which use specific properties 
of $\Z$.  In the next section, we shall construct pairs $(\S,\Theta)$ for $\Z$,
and it seems a much more delicate question as to whether this is tractable for
other groups $G$.
\end{remark}

\section{Examples}\label{eg}

This section gives examples of shift-invariant preduals arising from the methods of the previous section.  In particular, we show how the examples of Section \ref{haydon} can be realised in this way, we construct non-isometric shift-invariant preduals, and we construct shift-invariant preduals of $\ell_1(\Z)$ which are not isomorphic as Banach spaces to $c_0(\Z)$.  

For $k\in\mathbb N$, consider the additive semigroup $\S_k=\Z\times(\Z^+)^k\cup\{\infty\}$, where $\infty$ satisfies $\infty + \gamma = \gamma + \infty=\infty$ for all $\gamma\in\S_k$,
and $\mathbb Z^+=\{0,1,2,\cdots\}$.
We write the elements in $\S_k\setminus\{\infty\}$ as $\gamma=(\gamma_0,\gamma_1,\dots,\gamma_k)$ with $\gamma_0\in\Z$ and $\gamma_j\in\Z^+$ for $j=1,\dots,k$. The elements $e_i\in\S_k$, $i=1,2,\cdots,k$, with $1$ in the $i$-th co-ordinate and $0$'s elsewhere provide canonical semigroup generators for $\S_k\setminus\{\infty\}$ (depending on taste, one might also need to consider $-e_0$ as a generator).  We will subsequently discuss how to topologise $\S_k$ so as to turn it into a semitopological semigroup compactification of $\Z$.

For $a$ and $b$  in  $\ell_1(\mathcal S_k)$  we will denote  from now on the convolution of $a$ and $b$ by $a b$ instead of $a*b$, so that
\begin{equation}
ab(\gamma)=\sum_{\alpha,\beta\in\mathcal S_k, \alpha+\beta=\gamma}  a(\alpha)b(\beta),
\qquad \gamma\in\S_k.
\end{equation}
We consider $\mathbb Z$ naturally embedded in $\mathcal S_k$, by identifying $n\in\Z$ with $(n,0,0,\ldots0)\in \mathcal S_k$, and we will
 consider $\ell_1(\Z)$ as Banach subalgebra  of $\ell_1(\mathcal S_k)$. 
 We  also  consider the semigroup  $\mathcal S^0_k =(\mathbb Z^+)^k$ to be a subsemigroup of
 $\mathcal S_k$ by identifying $(\gamma_1,\ldots,\gamma_k)$ with   $(0,\gamma_1,\ldots,\gamma_k)$,
 for $\gamma_1,\gamma_2,\ldots \gamma_k\in \mathbb Z^+$. 
We will represent an element $\mu\in\ell_1(\mathcal S_k)$ as
\begin{equation}
\mu=\mu_\infty \delta_\infty+\sum_{\gamma\in\mathcal S_k^0}   \mu_\gamma \delta_\gamma 
\label{thomas_eq_one}
\end{equation}
where $\mu_\infty\in\mathbb C$ and 
$\mu_\gamma\in \ell_1(\mathbb Z)$ for $\gamma\in\mathcal S^0_k$.

A projection $\Theta:\ell_1(\S_k)\rightarrow\ell_1(\Z)$ which is also an algebra homomorphism is uniquely specified by the elements $a_i=\Theta(\delta_{e_i})$ for $i=1,\cdots,k$, as then
\begin{equation}
\Theta\Big(\mu_\infty\delta_\infty+\sum_{\gamma\in\mathcal S_k^0}   \mu_\gamma \delta_\gamma \Big)=
 \sum_{\gamma\in\mathcal S_k^0}   \mu_\gamma  \prod_{j=1}^k  a^{\gamma_j}_j,
 \quad\text{ for }   \mu_\infty\delta_\infty+\sum_{\gamma\in\mathcal S_k^0}   \mu_\gamma\delta_\gamma\in \ell_1(\mathcal S_k).
\end{equation}
As $\Theta$ is a projection and a homomorphism, it follows that $\Theta(\delta_\infty)
\delta_n = \Theta(\delta_\infty)$ for all $n\in\mathbb Z$, and so necessarily,
$\Theta(\delta_\infty)=0$.  Such a $\Theta$ is bounded if, and only if,
\begin{equation}
\max_{i=1,\dots,k}\sup_{m\in\mathbb N}\|a_i^m\|_1<\infty.
\end{equation}
To ensure that the kernel is weak$^*$-closed in $\ell_1(\S_k)$ (with respect to $C(\S_k)$ equipped with some suitable topology) we will need slightly stronger hypotheses.  

\begin{lemma}\label{kerclosed}
With the notation introduced above, suppose additionally that
\begin{equation}
\lim_{m\rightarrow\infty}\|a^m_i\|_\infty=0,\quad i=1,\dots,k.
\end{equation}
Then, regardless of the compact Hausdorff topology on $\S_k$, $\ker\Theta$ is weak$^*$-closed in $\ell_1(\S_k)$ with respect to $C(\S_k)$.
\end{lemma}
\begin{proof}
A useful result going back to Banach \cite[Page~124]{banach}, which can be easily
proved from the Krein-Smulian Theorem, shows that $\ker \Theta$ is weak$^*$-closed
if and only if $\{ \mu\in \ell_1(\S_k) : \|\mu\|\leq 1, \Theta(\mu)=0 \}$ is weak$^*$-closed. Thus it suffices to show that if $(\mu_\alpha)$ is a net in $\ker \Theta$ with
$\|\mu_\alpha\|\leq1$ for all $\alpha$, and $\mu_\alpha\rightarrow\mu$ weak$^*$,
then $\Theta(\mu)=0$.  For each $\alpha$, write
\begin{equation}
 \mu_\alpha =  \mu_\infty^{(\alpha)}\delta_\infty + \sum_{\gamma\in\S_k^0}
\mu^{(\alpha)}_\gamma\delta_\gamma,
\end{equation}
where each $\mu^{(\alpha)}_\gamma$ is regarded as lying in $\ell_1(\Z)\subset\ell_1(\S_k)$. Thus
$\|\mu_\alpha\|_1 = |\mu^{(\alpha)}_\infty| + \sum_\gamma \|\mu^{(\alpha)}_\gamma\|_1 \leq 1$.
Furthermore,
\begin{equation} 0 = \Theta(\mu_\alpha) = \sum_{\gamma\in\S_k^0}
\mu^{(\alpha)}_\gamma \prod_{j=1}^k a_j^{\gamma_j}. \label{eq.kerclosed.one}
\end{equation}

Fix $\varepsilon>0$ and choose $N$ such that $\|a_j^{n}\|_\infty<\varepsilon$ for
$n\geq N$ and each $j=1,\cdots,k$.  Partition $\S_k^0$ as $S_k^{0'} \cup S_k^{0''}$,
where $S_k^{0'}=\{(\gamma_1,\cdots,\gamma_k) : \gamma_i< N \ (i=1,\cdots,k) \}$.
By moving to a subnet, we may suppose that
$(\mu^{(\alpha)}_\gamma)$ is weak$^*$-convergent in $\ell_1(\S_k)=C(\S_k)^*$ for
each $\gamma\in\S_k^{0'}$. Then
\begin{align}
\Theta\Big( \lim_\alpha \sum_{\gamma\in\S_k^{0'}} \mu_\gamma^{(\alpha)} \delta_\gamma \Big)
&= \sum_{\gamma\in\S_k^{0'}} \Theta\Big( \big(\lim_\alpha \mu_\gamma^{(\alpha)}\big)\delta_\gamma \Big)
= \sum_{\gamma\in\S_k^{0'}} \Theta\Big( \lim_\alpha \mu_\gamma^{(\alpha)} \Big) \prod_j a_j^{\gamma_j}\nonumber \\
&= \Theta\Big( \lim_\alpha \sum_{\gamma\in\S_k^{0'}} \mu_\gamma^{(\alpha)}
   a_1^{\gamma_1}\cdots a_k^{\gamma_k} \Big).
\end{align}
Now partition $\S_k^{0''} = \S_k^1 \cup\cdots\cup \S_k^k$ where $\S_k^i = \{ (\gamma_1,\cdots,\gamma_k) : \gamma_i\geq N, \gamma_j<N \ (j<i) \}$.
Finally, define $\S_k^{i'} = \{ (\gamma_1,\cdots,\gamma_k) : \gamma_j< N (j<i) \}$,
so that $\S_k^{i'} = \{ \gamma \in \S^0_k : \gamma+Ne_i \in \S_k^i \}$.
Pick $t\in\Z$, and calculate that
\begin{align} \Theta(\mu)_t &=
\Theta\Big( \lim_\alpha \sum_{\gamma\in\S_k^{0'}} \mu_\gamma^{(\alpha)} \delta_\gamma
   + \lim_\alpha \sum_{j=1}^k \sum_{\gamma\in\S_k^j} \mu_\gamma^{(\alpha)} \delta_\gamma \Big)_t \nonumber\\
&= \Theta\Big( \lim_\alpha \sum_{\gamma\in\S_k^{0'}} \mu_\gamma^{(\alpha)}
   a_1^{\gamma_1}\cdots a_k^{\gamma_k} \Big)_t
   + \Theta\Big( \lim_\alpha \sum_{j=1}^k \sum_{\gamma\in\S_k^j} \mu_\gamma^{(\alpha)} \delta_\gamma \Big)_t \nonumber\\
&= - \Theta\Big( \lim_\alpha \sum_{j=1}^k \sum_{\gamma\in\S_k^j} \mu_\gamma^{(\alpha)}
   a_1^{\gamma_1}\cdots a_k^{\gamma_k} \Big)_t
   + \Theta\Big( \lim_\alpha \sum_{j=1}^k \delta_{Ne_j}
   \sum_{\gamma\in\S_k^{j'}} \mu_{\gamma+Ne_j}^{(\alpha)} \delta_\gamma \Big)_t\nonumber \\
&\qquad\qquad\qquad\Big[\text{as, by (\ref{eq.kerclosed.one}), }
\sum_{\gamma\in\S_k^{0'}} \mu_\gamma^{(\alpha)} \prod_j a_j^{\gamma_j} = - \sum_{\gamma\in\S_k^{0''}} \mu_\gamma^{(\alpha)} \prod_j a_j^{\gamma_j}
\text{ for all $\alpha$} \Big] \nonumber \\   
&= - \sum_{j=1}^k \Big( a_j^N\Theta\Big( \lim_\alpha \sum_{\gamma\in\S_k^{j'}} \mu_{\gamma+Ne_j}^{(\alpha)}
   a_1^{\gamma_1}\cdots a_k^{\gamma_k} \Big)\Big)_t
   + \sum_{j=1}^k \Big( a_j^N\Theta\Big( \lim_\alpha \sum_{\gamma\in\S_k^{j'}} \mu_{\gamma+Ne_j}^{(\alpha)} \delta_\gamma
   \Big)\Big)_t, \label{eg.1}
\end{align}
As $\Theta$ is bounded, we know that $K = \max_i \sup_m \|a^m_i\|_1 < \infty$, and so,
using $\ell_1$-$\ell_\infty$ duality, we obtain the estimates
\begin{align}
\big| \Theta(\mu)_t \big| &\leq 
   \sum_{j=1}^k \|a^N_j\|_\infty \|\Theta\| \liminf_\alpha \Big\|
   \sum_{\gamma\in\S_k^{j'}}
   \mu_{\gamma+Ne_j}^{(\alpha)} a_1^{\gamma_1}\cdots a_n^{\gamma_n} \Big\|_1 \nonumber
\\ &\qquad +
\sum_{j=1}^k\|a_j^N\|_\infty\|\Theta\|\liminf_\alpha\Big\|\sum_{\gamma\in\S_k^{j'}}\mu^{(\alpha)}_{\gamma+Ne_j}\delta_\gamma\Big\|_1 \\
& \leq k\varepsilon \|\Theta\| \sup_{j_1,\cdots,j_k} \|a_1^{j_1}\cdots a_n^{j_k}\|_1
   + k\varepsilon\|\Theta\| \\
& \leq k\varepsilon\|\Theta\|\big( K^k + 1 \big).
\end{align}
As $\varepsilon>0$ and $t$ were arbitrary, we conclude that $\Theta(\mu)=0$ as required.
\end{proof}

\begin{remark}
We do not know whether or not the condition $\lim_n \|a^n\|_\infty=0$ is necessary for  $\Theta$ -- together with some topology  on $\S_k$ which turns it into  a semitopological semigroup -- to
arise as part of a minimal pair inducing a shift-invariant
predual of $\ell_1(\mathbb Z)$.  Nevertheless we can conclude that 
$a^n$ converges weak$^*$ to 0 with respect to \emph{any} predual  arising in this fashion.

Indeed, whenever $\S_k$ provides a suitable semitopological semigroup compactification of $\Z$ (e.g. forms part of a pair inducing a shift-invariant predual for $\ell_1(\Z)$), then it follows that for each $i=1,\cdots k$, we have $n\cdot e_i\rightarrow\infty$ as $n\rightarrow\infty$. If this is not the case then we can find some net $(n_j)\subseteq\Z^+$ with $n_j\cdot e_i\rightarrow \gamma=(\gamma_0,\gamma_1,\cdots,\gamma_k)\in\S_k\setminus\{\infty\}$.  Clearly $(n_j)$ is unbounded, so passing to a further subnet, we may assume that $(n_j-(\gamma_i+1))\cdot e_i$ also converges, say to $\gamma'\in\S_k$.  Then
\begin{equation} \gamma' + ((\gamma_i+1)\cdot e_i) = \lim_j \big((n_j - (\gamma_i+1))\cdot e_i\big) + \big((\gamma_i+1)\cdot e_i\big) = \lim_j n_j\cdot e_i = \gamma, \end{equation}
but this means that $\gamma'_i + (\gamma_i+1) = \gamma_i$, which is impossible in $\Z^+$.
Thus our claim follows from Proposition~\ref{convergence} part 2.
As such, it is not a surprise that the previous result does not depend on the topology of $\S_k$.
\end{remark}

We will build suitable topologies on $\S_k$ by constructing suitable topologies on $\S_k\setminus\{\infty\}$ and then adding $\infty$ as a one-point compactification. The following is probably folklore, but we include a proof for completeness.

\begin{lemma}\label{lem:when_phi_semitop}
Let $\S$ be a locally compact Hausdorff semitopological semigroup.  Equip $\S\cup\{\infty\}$
with the one-point compactification topology, and let $\infty$ act as a semigroup zero.
Then $\S\cup\{\infty\}$ is a compact semitopological semigroup if, and only if, for each compact $K\subseteq\S$ and $\gamma\in \S$, the set $K-\gamma = \{ \gamma'\in \S:\gamma'+\gamma\in K \}$ is compact.
\end{lemma}
\begin{proof}
Translation by $\infty$ is obviously continuous.  So we need to show that if
$(s_\alpha)$ is a net in $\S$ converging to $\infty$, then for any $\gamma\in\S$,
also $s_\alpha + \gamma \rightarrow \infty$.  If the condition on compact sets holds,
then for any compact $K\subseteq \S$, we see that $s_\alpha+\gamma \in K$ if and only if
$s_\alpha \in K-\gamma$ which is compact.  So eventually $s_\alpha+\gamma$ is not in $K$; that is,
$s_\alpha+\gamma\rightarrow\infty$.

Conversely, suppose that the condition doesn't hold, so there is a compact set
$K\subseteq \S$ and $\gamma\in \S$ with $K-\gamma$ not compact.  Given compact sets $K_1,\cdots,K_n$, we must have $K-\gamma\not\subseteq K_1\cup\cdots\cup K_n$, otherwise $K-\gamma$ is a closed subset of the compact set $K_1\cup\cdots\cup K_n$ and so compact, contrary to hypothesis. So we can find a net $(s_\alpha)$ in
$K-\gamma$ such that $s_\alpha$ eventually leaves every compact set.  So $s_\alpha\rightarrow
\infty$ and yet $s_\alpha+\gamma\in K$ for all $\alpha$, so $s_\alpha+\gamma\not\rightarrow\infty$. Therefore $\S\cup\{\infty\}$ is not semitopological.
\end{proof}

Fix $k\geq 1$ and write $\T_k=\mathbb Z\times(\Z^+)^k$ so that $\S_k = \T_k\cup\{\infty\}$.  Suppose $J^{(1)},\cdots, J^{(k)}$ are infinite pairwise disjoint subsets of $\Z$. Our plan for constructing suitable topologies on $\T_k$ is to declare that limits $j^{(i)}_n$ from $J^{(i)}$ with $|n|\rightarrow\infty$ converge to the canonical semigroup generator $e_i\in\T_k$. This will give us neighbourhood bases of the semigroup generators. Neighbourhood bases of the remaining points of $\T_k$ are essentially forced upon us by the requirement that the semigroup operation be separately continuous. The only remaining issue is to extract suitable conditions on the sets $J^{(i)}$ which ensure the resulting topology on $\T_k$ is locally compact, Hausdorff and satisfies the `separate continuity at infinity' requirement of the previous lemma.

For each $\gamma=(\gamma_0,\gamma_1,\cdots,\gamma_k)\in\T_k$ and $n\in\mathbb N$, let $\mathcal V_{\gamma,n}$ be the subset consisting of those $\beta=(\beta_0,\cdots,\beta_k)\in\T_k$ with $\beta_i\leq\gamma_i$ for $i=1,\dots,k$ and
\begin{equation}\label{top.1}
\beta_0 = \gamma_0 + \sum_{i=1}^k\sum_{r=1}^{\gamma_i-\beta_i}j^{(i)}_r,
\end{equation}
where $(j^{(i)}_r)$ is a family such that:
\begin{enumerate}
\item $j^{(i)}_r \in J^{(i)}$ for each $r$;
\item $|j_r^{(i)}|\neq |j_s^{(k)}|$ when $(i,r)\neq (k,s)$ (this
condition is referred to as the $j_i^{(r)}$ having distinct absolute
values in the sequel);
\item $n < |j^{(i)}_1| < \cdots < |j^{(i)}_{\gamma_i-\beta_i}|$ for each $i$;
\end{enumerate}
In the following proof, we do not need to use the 2nd condition, but it will
be needed to make use of Definition~\ref{def:one} below.

Here we adopt the standard convention that the empty sum is $0$, so that, for example, if $\beta\in\mathcal V_{\gamma,n}$ has $\beta_i=\gamma_i$ for all $i=1,\cdots,k$, then $\beta_0=\gamma_0$ also. For the canonical semigroup generator $e_i$, the set $\mathcal V_{e_i,n}$ consists of $\{e_i\}\cup\{j\in J^{(i)}:|j|>n\}$ so, once we have shown that these sets provide a neighbourhood basis for $e_i$, it will follow that $j^{(i)}_\alpha\rightarrow e_i$ as $|j^{(i)}_\alpha|\rightarrow\infty$ through $J^{(i)}$.

For $\gamma=(\gamma_0,0,\cdots,0)$ we see that $\mathcal V_{\gamma,n} = \{\gamma\}$
for all $n$.  However, if $\gamma=(\gamma_0,\gamma_1,\cdots,\gamma_k)$ with some
$\gamma_i>0$ for $1\leq i\leq k$, then $\mathcal V_{\gamma,n}$ is infinite for all $n$.

\begin{lemma}\label{top.tech}
The sets $\mathcal V_{\gamma,n}$ describe an open neighbourhood basis at $\gamma$ for a topology on $\T_k$ with respect to which the semigroup operation is separately continuous. This topology is Hausdorff if, and only if, the following condition holds: for all $t\in\Z$ and $a_1,\cdots,a_k,b_1,\cdots,b_k\in\Z^+$, there exist $n\in\mathbb N$ such that if
\begin{equation}\label{top.3}
\sum_{i=1}^k\sum_{r=1}^{a_i} j^{(i)}_r = t + \sum_{i=1}^k\sum_{s=1}^{b_i}l^{(i)}_s
\end{equation}
for some $j^{(i)}_r,l^{(i)}_s\in J^{(i)}$ such that the $j^{(i)}_r$ have distinct absolute values, the $l^{(i)}_s$ have distinct absolute values and $|j^{(i)}_r|,|l^{(i)}_s|>n$, then $t=0$ and $a_i=b_i$ for $i=1,\cdots,k$.  In this case, the neighbourhoods $\mathcal V_{\gamma,n}$ are compact and for every compact set $K\subseteq\T_k$ and $\gamma\in\T_k$, the set $K-\gamma$ is compact.
\end{lemma}
\begin{proof}
We define $U\subseteq\T_k$ to be open if for each $\gamma\in U$, there exists
$n$ with $\mathcal V_{\gamma,n} \subseteq U$.  Then clearly $\emptyset$ and $\T_k$
are open, and unions of open sets are open.  If $U$ and $U'$ are open and
$\gamma\in U\cap U'$, then there are $n,n'\in\mathbb N$ with $\mathcal V_{\gamma,n}
\subseteq U$ and $\mathcal V_{\gamma,n'} \subseteq U'$, and thus
$\mathcal V_{\gamma,\max(n,n')} = \mathcal V_{\gamma,n}\cap\mathcal V_{\gamma,n'}
\subseteq U\cap U'$.  This shows that the intersection of two open sets will still be open.
So we do indeed have a topology on $\T_k$, where $\gamma\in\mathcal T_k$ has
neighbourhood basis $(\mathcal V_{\gamma,n})_{n\in\mathbb N}$.

Next we show that each $\mathcal V_{\gamma,n}$ is open. For $\gamma=(\gamma_0,\cdots,\gamma_k)\in\T_k$ and $n\in\mathbb N$, take $\beta\in \mathcal V_{\gamma,n}\setminus\{\gamma\}$ and let suitable $j^{(i)}_r$ be chosen so that (\ref{top.1}) holds.  Taking $n'=\max_{i,r}|j^{(i)}_r|$, we claim that $\mathcal V_{\beta,n'}\subset \mathcal V_{\gamma,n}$. Indeed, for $\alpha=(\alpha_0,\cdots,\alpha_k)\in \mathcal V_{\beta,n'}$ we can offset the
sum, and write
\begin{equation}\label{top.2}
\alpha_0=\beta_0+\sum_{i=1}^k\sum_{r=\gamma_i-\beta_i+1}^{\gamma_i-\alpha_i}j^{(i)}_r,
\end{equation}
for some additional $j_r^{(i)}$ with distinct absolute values and $|j^{(i)}_r|>n'$. The requirement that these additional $j_r^{(i)}$ have $|j^{(i)}_r|>n'$ ensures that all the $j^{(i)}_r$ have distinct absolute values and so (\ref{top.2}) combines with (\ref{top.1}) to show that $\alpha\in\mathcal V_{\gamma,n}$.   In this way each $\mathcal V_{\gamma,n}$ contains a neighbourhood of each of its points and so is open.  

To check that the addition is separately continuous, fix $\alpha,\gamma\in \T_k$ and a neighbourhood $\mathcal V_{\gamma+\alpha,n}$ of $\gamma+\alpha$. Then $\mathcal V_{\gamma,n}+\alpha\subseteq\mathcal V_{\gamma+\alpha,n}$.  Indeed, given $\beta\in \mathcal V_{\gamma,n}$, pick suitable $j^{(i)}_r$ such that (\ref{top.1}) holds.  Then the same
$j^{(i)}_r$ witness that $\beta+\alpha\in \mathcal V_{\gamma+\alpha,n}$.

The topology on $\T_k$ is Hausdorff if and only if for distinct $\gamma,\beta\in\T_k$ there exists some $n\in\mathbb N$ with $\mathcal V_{\gamma,n}\cap \mathcal V_{\beta,n}=\emptyset$.  Now, given $\alpha\in\mathcal V_{\gamma,n}\cap\mathcal V_{\beta,n}$ choose $j^{(i)}_r$ and $l^{(i)}_s$ such that the $j^{(i)}_r$'s and $l^{(i)}_s$'s have distinct absolute values, $|j_r^{(i)}|,|l^{(i)}_s|>n$ and
\begin{equation}
\alpha_0 = \gamma_0 + \sum_{i=1}^k\sum_{r=1}^{\gamma_i-\alpha_i}j^{(i)}_r
=\beta_0 + \sum_{i=1}^k\sum_{s=1}^{\beta_i-\alpha_i}l^{(i)}_s.
\end{equation}
Taking $t=\beta_0-\gamma_0$, $a_i=\gamma_i-\alpha_i$, $b_i=\beta_i-\alpha_i$, we see that the previous equation is equivalent to (\ref{top.3}) holding, and that the condition $t=0$ and $a_i=b_i$ for all $i$ is equivalent to $\gamma=\beta$. Thus $\T_k$ is Hausdorff if and only if the specified condition holds.

Now we establish compactness of the neighbourhoods $\mathcal V_{\gamma,n}$ by induction on $\sum_{i=1}^k\gamma_i$. When this sum is $0$, $\mathcal V_{\gamma,n}=\{\gamma\}$ which is certainly compact. Now fix $\gamma$ with $\sum_{i=1}^k\gamma_i>0$ and $n\in\mathbb N$. Take an open cover $\{U_\lambda:\lambda\in\Lambda\}$ of $\mathcal V_{\gamma,n}$.
There is some
$\lambda_0\in\Lambda$ with $\gamma\in U_{\lambda_0}$. Let
$n_0$ be minimal with $\mathcal V_{\gamma,n_0}\subseteq U_{\lambda_0}$
and note that if $n_0\leq n$, then $U_{\lambda_0}$ covers
$V_{\gamma,n}$; thus we may assume that $n_0>n$.
Given $\alpha\in\mathcal V_{\gamma,n}\setminus\mathcal V_{\gamma,n_0}$ choose
$j^{(i)}_r$ satisfying (\ref{top.1}), so $\alpha_0=\gamma_0+\sum_{i=1}^k\sum_{r=1}^{\gamma_i-\alpha_i}j^{(i)}_r$. Set $l(i)=|\{r:|j^{(i)}_r|\leq n_0\}|$. Now take $\beta_i=\gamma_i-l(i)$ for $i=1,\cdots,k$, and let
\begin{equation}
\beta_0=\gamma_0+\sum_{i=1}^k\sum_{r=1}^{l(i)}j^{(i)}_r.
\end{equation}
This construction ensures that $\alpha\in\mathcal V_{\beta,n}$. As $\alpha\notin\mathcal V_{\gamma,n_0}$, there is some $i_0\in\{1,\cdots,k\}$ with $l(i_0)\geq 1$ so that $\beta_{i_0}<\gamma_{i_0}$. As such the inductive hypothesis ensures that $\mathcal V_{\beta,n}$ is compact. Note too that $\beta$ is detemined by the values
of $l(i)$ in the range $0,\dots,\gamma_i$ and $(j^{(i)}_r)$ satisfying $n<|j^{(i)}_r|\leq n_0$ for $r=1,\cdots ,l(i)$ and so $\mathcal V_{\gamma,n}\setminus\mathcal V_{\gamma,n_0}$ is contained in a finite union of compact neighbourhoods $\mathcal V_{\beta,n}$.  Each of these is covered by a finite subcover of $\{U_\lambda:\lambda\in\Lambda\}$ and the union of these, together with $U_{\lambda_0}$, is a finite subcover, demonstrating that $\mathcal V_{\gamma,n}$ is compact. 

Finally take $K\subset\T_k$ compact and $\gamma\in\T_k$.  Suppose $\{U_\lambda:\lambda\in\Lambda\}$ is an open cover of $K-\gamma=\{\beta\in\T_k:\beta+\gamma\in K\}$.  Consider $\bigcup_{\lambda\in\Lambda}(U_\lambda+\gamma)$. This need
not cover $K$, but if $\alpha\in K$ is not in this union, then
$\alpha$ is not of the form $\beta+\gamma$ for any $\beta\in\mathcal T_k$, and so $\alpha_i<\gamma_i$ for some $i\in\{1,\cdots,k\}$.  Thus
\begin{equation}
\bigcup_{\lambda\in\Lambda}(U_\lambda+\gamma)\cup\bigcup_{\substack{\alpha\in
\mathcal T_k\\\exists i=1,\cdots,k:\alpha_i<\gamma_i}}\mathcal
V_{\alpha,1}
\end{equation}
covers $K$, and so has a finite subcover indexed by $\Lambda_0$ and
$\alpha^{(1)},\cdots,\alpha^{(m)}$ say.  Thus
\begin{equation}
\bigcup_{\lambda\in\Lambda_0}((U_\lambda+\gamma)-\gamma)\cup\bigcup_{s=1}^m(\mathcal
V_{\alpha^{(s)},1}-\gamma)
\end{equation}
covers $K-\gamma$. However, the sets in the second union are empty,
and $(U_\lambda+\gamma)-\gamma\subseteq U_\lambda$, so that
$\{U_\lambda:\lambda\in\Lambda_0\}$ is a finite subcover of the
original cover. Therefore $K-\gamma$ is compact.
\end{proof}

Combining Lemmas \ref{kerclosed}, \ref{lem:when_phi_semitop} and \ref{top.tech} with Theorem \ref{char-preduals} gives the following theorem, enabling us to produce a range of preduals.  We summarise this as a theorem.

\begin{theorem}\label{summary}
Fix $k\in\mathbb N$ and let $J^{(1)},\cdots,J^{(k)}$ be infinite pairwise disjoint subsets of $\Z$ satisfying the technical condition in Lemma \ref{top.tech} and let $\T_k$ have the topology given by the neighbourhoods $\mathcal V_{\gamma,n}$ of $\gamma\in\T_k$.  Let $a_1,\cdots,a_k\in\ell_1(\Z)$ be $\ell_1(\Z)$-power bounded elements which satisfy $\|a^m_i\|_\infty\rightarrow 0$ as $m\rightarrow\infty$ for each $i=1,\cdots,k$. Define a bounded projection $\Theta:\ell_1(\S_k)\rightarrow\ell_1(\Z)$ which is also an algebra homomorphism by $\Theta(\delta_{e_i})=a_i$, where $e_i$ is the $i$-th semigroup generator of $\S_k$.  Then $\ker\Theta$ is weak$^*$-closed in $\ell_1(\S_k)$ (with respect to $C(\S_k)$) and $F={}^\perp\ker\Theta\subset C(\S_k)$ restricts to $\Z$ to define a shift-invariant predual $E$ of $\ell_1(\Z)$.  The resulting weak$^*$-topology on $\ell_1(\Z)$ is such that $\delta_n\rightarrow a_i$ as $|n|\rightarrow\infty$ through $J^{(i)}$.
\end{theorem}

To produce examples we need to provide suitable sets $J^{(i)}$ and elements $a_i$.
\begin{definition}\label{def:one}
Let $J\subset\Z$ be infinite. Say that $J$ is \emph{additively sparse} if, given $t\in\Z$ and $r,s\in\Z^+$, there exists $n\in\mathbb N$ such that if
$$
j_1+\cdots+j_r=l_1+\cdots+l_s+t
$$
for some $j_i,l_i\in J$ with $n<|j_1|<|j_2|<\cdots<|j_r|$ and $n<|l_1|<\cdots<|l_s|$, then $t=0$, $r=s$ and $j_1=l_1,\cdots,j_r=s_r$.
\end{definition}

If $J$ is additively sparse and $J^{(1)}, \cdots, J^{(k)}$ are infinite pairwise disjoint subsets of $J$, then the condition of Lemma~\ref{top.tech} is satisfied (perform a simple induction on $k$).  Using $m$-ary expansions, it is easily seen that $\{m^n:n\in\mathbb N\}$ is additively sparse for each $m>0$.  It is also straight-forward to show that $\{\pm (m!):m>0\}$ is additively sparse.

\begin{example}
Taking $k=1$, $J^{(1)}=\{2^n:n\in\mathbb N\}$ and $a_1=\lambda^{-1}\delta_0$ for some $\lambda\in\mathbb C$ with $|\lambda|>1$ gives the preduals $F^{(\lambda)}$ considered in Section~\ref{haydon}.  Indeed it is routine to check that the $x_0$ from Section \ref{haydon} lies in $^\perp\ker\Theta$ and since the predual $F^{(\lambda)}$ is the smallest closed shift-invariant subspace containing $x_0$, it follows that $F^{(\lambda)}\subset{}^\perp\ker\Theta$. Since an inclusion of concrete preduals implies that these preduals are equal (see the comment after Lemma \ref{concreteequal}) $F^{(\lambda)}={}^\perp\ker\Theta$.
\end{example}

\begin{theorem}\label{not_co_thm}
There exists a shift-invariant predual $E$ of $\ell_1(\Z)$ such that, as
a Banach space, $E$ is not isomorphic to $c_0$.
\end{theorem}
\begin{proof}
Let $k=1$, $J^{(1)}$ be any additively sparse set and $a_1=\frac{1}{2}(\delta_0+\delta_1)$.
Certainly $\|a^m_1\|_1=1$ for all $m\in\mathbb N$. We can approximate $\|a^m_1\|_\infty$
by Stirling's formula to estimate the central binomial coefficient.  Indeed
$a_1^m=2^{-m}\sum_{i=0}^m\binom{m}{i}\delta_i$ so that
\begin{equation}
\|a_1^m\|_\infty=\frac{1}{2^m}\binom{m}{\lfloor m/2\rfloor}.
\end{equation}
Taking $m=2n$, we have
\begin{equation}
\|a_1^{2n}\|_\infty=\frac{1}{2^{2n}}\frac{(2n)!}{(n!)^2}\sim\frac{1}{4^n}\frac{4^n}{\sqrt{\pi n}}\rightarrow 0,\quad \text{as }n\rightarrow\infty.
\end{equation}
Thus $a_1$ satisfies the requirements of Theorem~\ref{summary} and we can obtain shift-invariant preduals $E$ with $\delta_n\rightarrow a_1$ as $|n|\rightarrow\infty$ through additively sparse sets.  Since $\|a_1^m\|_1=1$ for all $m\in\mathbb N$,
Proposition~\ref{notco} shows that these preduals are not isomorphic as Banach spaces to $c_0$.
\end{proof}

\begin{remark}
The shift-invariant predual constructed in the previous theorem has Szlenk index
$\omega^2$.  Indeed, the proof shows that the Szlenk index must be larger than
$\omega$.  However, by \cite{BP}, $\S_1$ is homeomorphic to $[0,\omega^\omega]$
and thus $C(\S_1)$ has Szlenk index $\omega^2$.  By \cite[Corollary~3.10]{ajo},
the Szlenk index is always of the form $\omega^\alpha$, and so the only possibility
is that $E$ has Szlenk index $\omega^2$.
\end{remark}

\begin{proposition}
Recall that $\ell_1(\Z)$ carries a natural involution, where $\delta_n^* = \delta_{-n}$.
Let $E$ be a predual arising from Theorem~\ref{summary} where each $J^{(i)}$ is
\emph{symmetric} in the sense that $j\in J^{(i)}$ if and only if $-j\in J^{(i)}$,
and $a_i^*=a_i$ for each $i$.  Then $E$ makes the involution on $\ell_1(\Z)$ weak$^*$-continuous.
\end{proposition}
\begin{proof}
As the $J^{(i)}$ are symmetric, the basic neighbourhoods $\mathcal V_{\gamma,n}$ are invariant under the map 
\begin{equation}
\phi:\S_k\rightarrow\S_k;\quad (\beta_0,\beta_1,\cdots,\beta_k)\mapsto (-\beta_0,\beta_1,\cdots,\beta_k),\quad \infty\mapsto\infty
\end{equation}
and so $\phi$ is continuous. The involution $^*$ on $\ell_1(\Z)$ extends to $\ell_1(\S_k)$ by 
\begin{equation}
\sum_{\gamma\in\S_k}c_\gamma\delta_\gamma\mapsto\sum_{\gamma\in\S_k}\overline{c_\gamma}\delta_{\phi(\gamma)},
\end{equation}
and the assumption that $a_i^*=a_i$ for $i=1,\cdots,k$ gives $\Theta(\mu^*)=\Theta(\mu)^*$ for $\mu\in \ell_1(\S_k)$. Since $\phi$ is continuous, we also obtain an involution $^\dagger$ on $C(\S_k)$ by $f^\dagger(\gamma)=\overline{f(\phi(\gamma))}$.
so that
\begin{equation}
\ip{c^*}{f}=\overline{\ip{c}{f^\dagger}},\quad c\in\ell_1(\S_k),\ f\in C(\S_k).
\end{equation}
Now suppose $(b_i)$ is a net in $\ell_1(\Z)$ such that $b_i\rightarrow b$ and $b_i^*\rightarrow c$ in the weak$^*$-topology on $\ell_1(\Z)$ induced by $E$. Passing to a subnet, we may assume that $b_i\rightarrow \mu\in\ell_1(\S_k)$ and $b_i^*\rightarrow\nu\in\ell_1(\S_k)$ in the weak$^*$-topology induced by $C(\S_k)$, so that $b=\Theta(\mu)$ and $c=\Theta(\nu)$. Now
\begin{equation}
\ip{\mu}{f}=\lim_i\ip{b_i}{f}=\lim_i\overline{\ip{b_i^*}{f^\dagger}}=\overline{\ip{\nu}{f^\dagger}}=\ip{\nu^*}{f},\quad f\in C(\S_k),
\end{equation}
so that $\nu^*=\mu$ and $\mu^*=\nu$.  Thus 
\begin{equation}
b^*=\Theta(\mu)^*=\Theta(\mu^*)=\Theta(\nu)=c,
\end{equation}
and the involution is weak$^*$-continuous.
\end{proof}

Examples of this phenomena can be obtained by using the symmetric additively sparse set $\{\pm (n!):n>0\}$.

\begin{example}\label{not_iso_eg}
We thank Yemon Choi, \cite{mof}, for pointing us to this example.  Let
$a=5^{-1/2}(\delta_0+\delta_1-\delta_2) \in \ell_1(\Z)$, so that $\|a\|_1 = 3/\sqrt5>1$.
In \cite[Page~39]{newman}, Newman shows that $a$ is power bounded.  The Fourier transform
of $a$ is $f(z)=5^{-1/2}(1+z-z^2)$ for $z\in\mathbb T$.  Thus, for $z=e^{i\theta}$,
\begin{equation} |f(z)| = 5^{-1/2}\big|z^{-1}+1-z\big| = 5^{-1/2} \big| 1-2i\sin\theta \big|
= \big( 1-\textstyle\frac45\cos^2\theta\big)^{1/2}. \end{equation}
Thus, for any $\varepsilon>0$, if $n$ is sufficiently large, then $|f^n|<\varepsilon$
except on intervals of length at most $\varepsilon$ about the points $\theta=\pi/2, 3\pi/2$.
Thus $\lim_n \int_{\mathbb T} f(z)^n z^m \ dz = 0$ uniformly in $m\in\Z$, which shows that
$\lim_n \|a^n\|_\infty = 0$.

We can hence apply our theorem with $k=1$ and $J$ being any additively sparse set.
The resulting predual $E\subseteq\ell_\infty(\Z)$ is not \emph{isometric},
in the sense that the map $\iota_E:\ell_1(\Z)\rightarrow E^*$ is only an isomorphism,
not an isometric isomorphism.  This follows, as for $x\in E$,
\begin{equation} \lim_{n\in J} \ip{\iota_E(\delta_n)}{x} =
 \lim_{n\in J} \ip{x}{\delta_n} = \ip{x}{a} = \ip{\iota_E(a)}{x}. \end{equation}
So if $\iota_E$ were an isometry, we would have that $1 < \|a\| = \|\iota_E(a)\|
 \leq \limsup_n \|\iota_E(\delta_n)\| = 1$, a contradiction.
\end{example}

\section{Questions}

We end the paper with a range of open questions regarding these preduals.

\begin{enumerate}
\item Describe all possible semigroups $\S$ arising as part of a minimal pair
  inducing a shift-invariant predual of $\ell_1(\Z)$.
\item What are the Banach space isomorphism classes of shift-invariant preduals?
\item What is the Banach space isomorphism class of the shift-invariant predual
  constructed in Theorem~\ref{not_co_thm}?
\item For any countable ordinal $\alpha$, does there exist a shift-invariant predual
  with Szlenk index at least $\alpha$?
\item Characterise those $a\in\ell_1(\Z)$ which occur as weak$^*$-limit points of
  $\{\delta_n:n\in\Z\}$.  In particular, is the condition $\lim_n \|a^n\|_\infty=0$
  necessary as well as sufficient?
\item The concrete shift-invariant preduals $F^{(\lambda)}$ of Section~\ref{haydon}
  are \emph{cyclic} in that they are the minimal closed, shift-invariant subspaces
  containing the specified element $x_0$.  Characterise the cyclic shift-invariant
  preduals of $\ell_1(\Z)$.
\end{enumerate}

\noindent Matthew Daws  \\
School of Mathematics  \\
University of Leeds  \\
Leeds LS2 9JT  \\
United Kingdom  \\
Email: \texttt{matt.daws@cantab.net}

\medskip

\noindent Richard Haydon\\
Brasenose College\\
Radcliffe Square\\
Oxford OX1 4AJ\\
United Kingdom  \\
Email: \texttt{richard.haydon@bnc.ox.ac.uk}

\medskip

\noindent Thomas Schlumprecht \\
Department of Mathematics\\
Mailstop 3368\\
Texas A\&M University\\
College Station, TX 77843-3368\\
United States of America\\
Email: \texttt{schlump@math.tamu.edu}

\medskip

\noindent Stuart White  \\
School of Mathematics and Statistics  \\
University of Glasgow  \\
Glasgow G12 8QW  \\
United Kingdom  \\
Email: \texttt{stuart.white@glasgow.ac.uk}


\begin{thebibliography}{9}

\bibitem{ajo} D. Alspach, R. Judd, E. Odell,
   ``The {S}zlenk index and local {$l_1$}-indices'',
   \emph{Positivity} 9 (2005) 1--44.

\bibitem{ah} S. Argyros, R. Haydon,
   ``A hereditarily indecomposable $\mathcal L_\infty$-space that solves the scalar-plus-compact problem'', Acta. Math., to appear. arXiv:0903.3921v2 [math.FA]

\bibitem{banach} S. Banach,
   ``Th\'eorie des Op\'erations Lin\'eaires'',
   (Subwencji Funduszu Kultury Narodowej, 1932).

\bibitem{ben} Y. Benyamini,
   ``Separable $G$ spaces are isomorphic to $C(K)$ spaces'',
   \emph{Israel J. Math.} 14 (1973) 287--293. 

\bibitem{bl} Y. Benyamini, J. Lindenstrauss,
   ``A predual of $l_{1}$ which is not isomorphic to a $C(K)$ space'',
   \emph{Israel J. Math.} 13 (1972) 246--254.

\bibitem{semi} J.\,F. Berglund, H.\,D. Junghenn, P. Milnes, ``Analysis on semigroups'',
   (John Wiley \& Sons, New York, 1989).

\bibitem{BP} C. Bessaga, C. A. Pe{\l}czy{\'n}ski,
   ``Spaces of continuous functions. {IV}. {O}n isomorphical
              classification of spaces of continuous functions'',
   \emph{Studia Math.} 19 (1960) 53--62.

\bibitem{bd} J. Bourgain, F. Delbaen,
   ``A class of special ${\cal L}_{\infty }$ spaces'',
   \emph{Acta Math.} 145 (1980) 155--176.

\bibitem{mof} Y. Choi, posting on MathOverFlow.net,
   \texttt{http://mathoverflow.net/questions/37305/\\odd-element-of-l1-group-algebra-of-the-integers/37336$\sharp$37336}

\bibitem{cpss} E. Christensen, F. Pop, A.\,M. Sinclair, R.\,R. Smith,
   ``Hochschild cohomology of factors with property $\Gamma$'',
   \emph{Ann. of Math. (2)} 158 (2003) 6350--659.

\bibitem{DGH} H.\,G. Dales, F. Ghahramani, A. Ya. Helemskii,
   ``The amenability of measure algebras'',
   \emph{J. London Math. Soc.} 66 (2002) 213--226.

\bibitem{daws} M. Daws,
   ``Dual Banach algebras: representations and injectivity'',
   \emph{Studia Math.} 178 (2007) 231--275.

\bibitem{dlpw} M. Daws, H. Le Pham, S. White,
   ``Conditions implying the uniqueness of the weak${}^*$-topology on certain group algebras'',
   \emph{Houston J. Math.} 35 (2009) 253--276.

\bibitem{semipaper} M. Daws, H. Le Pham, S. White,
   ``Preduals of semigroup algebras'',
   \emph{Semigroup Forum} 80 (2010) 61--78. 

\bibitem{fos} D. Freeman, E. Odell, Th. Schlumprecht,
   ``The universality of $\ell_1$ as a dual space'',
   \emph{Math. Ann.}, to appear.  See arXiv:0904.0462v2 [math.FA]

\bibitem{gulick} S. L. Gulick, ``Commutativity and ideals in the biduals of topological algebras.''  \emph{Pacific J. Math.} 18  (1966) 121--137.

\bibitem{joh} B.~E.~Johnson, ``Separate continuity and measurability'', \emph{Proc.  Americ. Math. Soc.}
{\bf 20} (1969) 420 -- 422.

\bibitem{kai} S. Kaijser,
   ``On {B}anach modules. {I}'',
   \emph{Math. Proc. Cambridge Philos. Soc.} 90 (1981) 423--444.

\bibitem{lan} G. Lancien,
   ``Dentability indices and locally uniformly convex renormings'',
   \emph{Rocky Mountain J. Math.} 23 (1993) 635--647.

\bibitem{lau} A.\,T.-M. Lau, R.\,J. Loy,
   ``Weak amenability of {B}anach algebras on locally compact groups'',
   \emph{J. Funct. Anal.} 145 (1997) 175--204.

\bibitem{newman} D.\,J. Newman,
   ``Homomorphisms of {$l_{+}$}'',
   \emph{Amer. J. Math.} 91 (1969) 37--46.

\bibitem{pel} A. Pe{\l}czy\'nski,
   ``On the isomorphism of the spaces $m$ and $M$'',
   \emph{Bull. Acad. Polon. Sci. S�r. Sci. Math. Astr. Phys.} 6 (1958) 695--696. 

\bibitem{ros} H.\,P. Rosenthal,
   ``The {B}anach spaces {$C(K)$}'',
   in ``Handbook of the geometry of {B}anach spaces, {V}ol.\ 2'',
   pp. 1547--1602 (North-Holland, Amsterdam, 2003).

\bibitem{Runde2} V. Runde,
   ``Amenability for dual {B}anach algebras'',
   \emph{Studia Math.} 148 (2001) 47--66.

\bibitem{runde} V. Runde,
   ``Connes-amenability and normal, virtual diagonals for measure algebras. {I}'',
   \emph{J. London Math. Soc.} 67 (2003) 643--656.

\bibitem{rup} W.\,A.\,F. Ruppert,
   ``On signed $a$-adic expansions and weakly almost periodic functions'',
   \emph{Proc. London Math. Soc.} 63 (1991) 620--656.

\bibitem{sakai} S. Sakai,
   ``$C\sp*$-algebras and $W\sp*$-algebras'',
   (Springer-Verlag, New York-Heidelberg, 1971).

\bibitem{sam} C. Samuel,
   ``Indice de {S}zlenk des {$C(K)$} ({$K$} espace topologique compact d\'enombrable)'',
   in ``Seminar on the geometry of {B}anach spaces, {V}ol. {I}, {II} ({P}aris, 1983)'',
   \emph{Publ. Math. Univ. Paris VII} 18 (1984) 81--91.

\bibitem{ss2} A.\,M. Sinclair, R.\,R. Smith,
   ``A survey of Hochschild cohomology for von Neumann algebras'',
   \emph{Contemp. Math.} 365 (2004) 383-400.

\bibitem{szlenk} W. Szlenk,
   ``The non-existence of a separable reflexive {B}anach space
              universal for all separable reflexive {B}anach spaces'',
   \emph{Studia Math.} 30 (1968) 53--61.

\bibitem{young}  N.\,J. Young,
   ``Periodicity of functionals and representations of normed
              algebras on reflexive spaces'',
   \emph{Proc. Edinburgh Math. Soc.} 20 (1976/77) 99--120.

\end{thebibliography}
\end{document}